\documentclass[12pt]{amsart}
\usepackage[utf8]{inputenc}
\usepackage{amssymb}
\usepackage{hyperref}
\usepackage[final]{showkeys} %[final] is the option that removes them

\input xy
\xyoption{all}

\theoremstyle{definition}
\newtheorem{mydef}{Definition}[section]
\newtheorem{lem}[mydef]{Lemma}
\newtheorem{thm}[mydef]{Theorem}
\newtheorem{conjecture}[mydef]{Conjecture}
\newtheorem{cor}[mydef]{Corollary}
\newtheorem{claim}[mydef]{Claim}
\newtheorem{question}[mydef]{Question}
\newtheorem{hypothesis}[mydef]{Hypothesis}
\newtheorem{prop}[mydef]{Proposition}
\newtheorem{defin}[mydef]{Definition}

\newtheorem{remark}[mydef]{Remark}
\newtheorem{notation}[mydef]{Notation}
\newtheorem{fact}[mydef]{Fact}

\newcommand{\fct}[2]{{}^{#1}#2}

\newcommand{\ba}{\bar{a}}
\newcommand{\bb}{\bar{b}}

\newcommand{\Ksatpp}[2]{{#1}^{#2\text{-sat}}}
\newcommand{\Ksatp}[1]{\Ksatpp{\K}{#1}}

\newcommand{\cf}[1]{\text{cf} (#1)}
\newcommand{\seq}[1]{\langle #1 \rangle}
\newcommand{\rest}{\upharpoonright}

\newcommand{\s}{\mathfrak{s}}

\newcommand{\ts}{\mathfrak{t}}

\newcommand{\overunder}[3]{\underset{#2}{\overset{#3}{#1}}}
\newcommand{\perpp}[2]{\overunder{\perp}{#1}{#2}}
\newcommand{\wkperpp}[1]{\perpp{\text{wk}}{#1}}
\newcommand{\wkperp}{\wkperpp{}}

\newcommand{\Ktuqp}[1]{\K_{#1}^{3, \text{uq}}}
\newcommand{\Ktuq}{\Ktuqp{}}

\newcommand{\K}{\mathbf{K}}

\newcommand{\leap}[1]{\le_{#1}}
\newcommand{\ltap}[1]{<_{#1}}
\newcommand{\gtap}[1]{>_{#1}}
\newcommand{\geap}[1]{\ge_{#1}}

\newcommand{\lta}{\ltap{\K}}
\newcommand{\lea}{\leap{\K}}
\newcommand{\gta}{\gtap{\K}}
\newcommand{\gea}{\geap{\K}}

\newbox\noforkbox \newdimen\forklinewidth
\forklinewidth=0.3pt \setbox0\hbox{$\textstyle\smile$}
\setbox1\hbox to \wd0{\hfil\vrule width \forklinewidth depth-2pt
 height 10pt \hfil}
\wd1=0 cm \setbox\noforkbox\hbox{\lower 2pt\box1\lower
2pt\box0\relax}

\def\unionstick{\mathop{\copy\noforkbox}\limits}
\newcommand{\nf}{\unionstick}

\setbox0\hbox{$\textstyle\smile$}
\setbox1\hbox to \wd0{\hfil{\sl /\/}\hfil} \setbox2\hbox to
\wd0{\hfil\vrule height 10pt depth -2pt width
               \forklinewidth\hfil}
\newbox\doesforkbox
\setbox\doesforkbox\hbox{\lower 0pt\box1 \lower
2pt\box2\lower2pt\box0\relax}
\def\nunionstick{\mathop{\copy\doesforkbox}\limits}
\newcommand{\nnf}{\nunionstick}

\newcommand{\nfs}[4]{#2 \nf_{#1}^{#4} #3}
\newcommand{\nnfs}[4]{#2 \nnf_{#1}^{#4} #3}

\def\1nf{\unionstick^{(1)}}

\def\2nf{\unionstick^{(2)}}
\def\3nf{\unionstick^{(3)}}

\newcommand{\gtp}{\text{gtp}}

\newcommand{\gS}{\text{gS}}

\newcommand{\Sbs}{\gS^\text{bs}}

\newcommand{\hanf}[1]{h (#1)}

\newcommand{\ehanf}[1]{\beth_{\left(2^{#1}\right)^+}}

\newcommand{\Ll}{\mathbb{L}}

\newcommand{\WGCH}{\text{WGCH}}

\newcommand{\F}{\mathcal{F}}
\newcommand{\LS}{\text{LS}}

\newcommand{\textpow}[2]{\text{#1}^{#2}}
\newcommand{\textplus}[1]{\textpow{#1}{+}}

\newcommand{\succp}{\textplus{successful}}

\title[Categoricity in a good frame]{Downward categoricity from a successor inside a good frame}
\date{\today}

\author{Sebastien Vasey}
\email{sebv@cmu.edu}
\urladdr{http://math.cmu.edu/\textasciitilde svasey/}
\address{Department of Mathematical Sciences, Carnegie Mellon University, Pittsburgh, Pennsylvania, USA}
\thanks{This material is based upon work done while the second author was supported by the Swiss National Science Foundation under Grant No.\ 155136.}

\begin{document}

\begin{abstract}
  In the setting of abstract elementary classes (AECs) with amalgamation, Shelah has proven a downward categoricity transfer from categoricity in a successor and Grossberg and VanDieren have established an upward transfer assuming in addition a locality property for Galois types that they called tameness.
  
  We further investigate categoricity transfers in tame AECs. We use orthogonality calculus to prove a downward transfer from categoricity in a successor in AECs that have a good frame (a forking-like notion for types of singletons) on an interval of cardinals:

  \begin{thm}\label{main-technical-thm-abstract}
    Let $\K$ be an AEC and let $\LS (\K) \le \lambda < \theta$ be cardinals. If $\K$ has a type-full good $[\lambda, \theta]$-frame and $\K$ is categorical in both $\lambda$ and $\theta^+$, then $\K$ is categorical in all $\mu \in [\lambda, \theta]$.
  \end{thm}

  We deduce improvements on the threshold of several categoricity transfers that do not mention frames. For example, the threshold in Shelah's transfer can be improved from $\beth_{\beth_{\left(2^{\LS (\K)}\right)^+}}$ to $\beth_{\left(2^{\LS (\K)}\right)^+}$ assuming that the AEC is $\LS (\K)$-tame. The successor hypothesis can also be removed from Shelah's result by assuming in addition either that the AEC has primes over sets of the form $M \cup \{a\}$ or (using an unpublished claim of Shelah) that the weak generalized continuum hypothesis holds.
\end{abstract}

\maketitle

\tableofcontents

\section{Introduction}

\subsection{Motivation and history}

In his two volume book \cite{shelahaecbook, shelahaecbook2} on classification theory for abstract elementary classes (AECs), Shelah introduces the notion of a \emph{good $\lambda$-frame} \cite[II.2.1]{shelahaecbook}. Roughly, a good $\lambda$-frame is a local notion of independence for types of length one over models of size $\lambda$. The independence notion satisfies basic properties of forking in a superstable first-order theory. Good frames are the central concept of the book. In Chapter II and III, Shelah discusses the following three questions regarding frames:

\begin{question}\label{frame-q}
\begin{enumerate} \
\item\label{frame-q-1} Given an AEC $\K$, when does there exist a good $\lambda$-frame $\s$ whose underlying AEC $\K_{\s}$ is $\K_\lambda$ (or some subclass of saturated models in $\K_\lambda$)?
\item\label{frame-q-2} Given a good $\lambda$-frame, under what conditions can it be extended to a good $\lambda^+$-frame?
\item\label{frame-q-3} Once one has a good frame, how can one prove categoricity transfers?
\end{enumerate}
\end{question}

Shelah's answers (see for example II.3.7, III.1, and III.2 in \cite{shelahaecbook}) involve a mix of set-theoretic hypotheses (such as the weak generalized continuum hypothesis: $2^{\theta} < 2^{\theta^+}$ for all cardinals $\theta$) and strong local model-theoretic hypotheses (such as few models in $\lambda^{++}$). While Shelah's approach is very powerful (for example in \cite[Chapter IV]{shelahaecbook}, Shelah proves the eventual categoricity conjecture in AECs with amalgamation assuming some set-theoretic hypotheses, see more below), most of his results do not hold in ZFC.

An alternate approach is to make \emph{global} model-theoretic assumptions. In \cite{tamenessone}, Grossberg and VanDieren introduced \emph{tameness}, a locality property which says that Galois types are determined by their small restrictions. In \cite{ext-frame-jml}, Boney showed that in an AEC which is $\lambda$-tame for types of length two and has amalgamation, a good $\lambda$-frame can be extended to all models of size at least $\lambda$ (we call the resulting object a good $(\ge \lambda)$-frame, and similarly define good $[\lambda, \theta]$-frame for $\theta > \lambda$ a cardinal). In \cite{tame-frames-revisited-v5}, tameness for types of length two was improved to tameness for types of length one. In particular, the answer to Question \ref{frame-q}.(\ref{frame-q-2}) is always positive in tame AECs with amalgamation. As for existence (Question \ref{frame-q}.(\ref{frame-q-1})), we showed in \cite{ss-tame-jsl} how to build good frames in tame AECs with amalgamation assuming categoricity in a cardinal of high-enough cofinality. Further improvements were made in \cite{indep-aec-apal, bv-sat-v3, vv-symmetry-transfer-v3}. This gives answers to Questions \ref{frame-q}.(\ref{frame-q-1}),(\ref{frame-q-2}) in tame AECs with amalgamation:

\begin{fact}\label{existence-extension-tame}
  Let $\K$ be an AEC with amalgamation and let $\lambda \ge \LS (\K)$ be such that $\K$ is $\lambda$-tame.
  
  \begin{enumerate}
  \item\cite[Corollary 6.9]{tame-frames-revisited-v5} If there is a good $\lambda$-frame $\s$ with $\K_{\s} = \K_\lambda$, then $\s$ can be extended to a good $(\ge \lambda)$-frame (with underlying class $\K$).
  \item\cite[Corollary 6.14]{vv-symmetry-transfer-v3} If $\K$ has no maximal models and is categorical in some $\mu > \lambda$, then there is a type-full good $\lambda^+$-frame with underlying class the Galois saturated models of $\K$ of size $\lambda^+$.
  \end{enumerate}
\end{fact}

\subsection{Categoricity in good frames}

In this paper, we study Question \ref{frame-q}.(\ref{frame-q-3}) in the global setting: assuming the existence of a good frame together with some global model-theoretic properties, what can we say about the categoricity spectrum? From the two results above, it is natural to assume that we are already working inside a type-full good $(\ge \lambda)$-frame (this implies properties such as $\lambda$-tameness and amalgamation). It is then known how to transfer categoricity with the additional assumption that the class has primes over sets of the form $M \cup \{a\}$. This has been used to prove Shelah's eventual categoricity conjecture for universal classes \cite{ap-universal-v9, categ-universal-2-v2}.

\begin{defin}[III.3.2 in \cite{shelahaecbook}]\label{prime-def}
  An AEC $\K$ \emph{has primes} if for any nonalgebraic Galois type $p \in \gS (M)$ there exists a triple $(a, M, N)$ such that $p = \gtp (a / M; N)$ and for every $N' \in \K$, $a' \in |N'|$, such that $p = \gtp (a' / M; N')$, there exists $f: N \xrightarrow[M]{} N'$ with $f (a) = a'$.
\end{defin}

\begin{fact}[Theorem 2.16 in \cite{categ-primes-v3}]
  Assume that there is a type-full good $[\lambda, \theta)$-frame on the AEC $\K$. Assume that $\K$ has primes and is categorical in $\lambda$. If $\K$ is categorical in \emph{some} $\mu \in (\lambda, \theta]$, then $\K$ is categorical in \emph{all} $\mu' \in (\lambda, \theta]$.
\end{fact}

What if we do not assume existence of primes? The main result of this paper is a downward categoricity transfer for global good frames categorical in a successor:

\textbf{Theorem \ref{good-categ-transfer}.}
Assume that there is a type-full good $[\lambda, \theta]$-frame on the AEC $\K$. Assume that $\K$ is categorical in $\lambda$. If $\K$ is categorical in $\theta^+$, then $\K$ is categorical in \emph{all} $\mu \in [\lambda, \theta]$.

The proof of Theorem \ref{good-categ-transfer} develops orthogonality calculus in this setup (versions of some of our results on orthogonality have been independently derived by Villaveces and Zambrano \cite{viza}). We were heavily inspired from Shelah's development of orthogonality calculus in successful good $\lambda$-frames \cite[Section III.6]{shelahaecbook}, and use it to define a notion of unidimensionality similar to what is defined in \cite[Section III.2]{shelahaecbook}. We show unidimensionality in $\lambda$ is equivalent to categoricity in $\lambda^+$ and use orthogonality calculus to transfer unidimensionality across cardinals. While we work in a more global setup than Shelah's, we do \emph{not} assume that the good frames we work with are successful \cite[Definition III.1.1]{shelahaecbook}, so we do \emph{not} assume that the forking relation is defined for types of models (it is only defined for types of elements). To get around this difficulty, we use the theory of \emph{independent sequences} introduced by Shelah for good $\lambda$-frames in \cite[Section III.5]{shelahaecbook} and developed in \cite{tame-frames-revisited-v5} for global good frames. 

\subsection{Hypotheses of the main theorem}

Let us discuss the hypotheses of Theorem \ref{good-categ-transfer}. We are assuming that the good frame is \emph{type-full}: the basic types are all the nonalgebraic types. This is a natural assumption to make if we are only interested in tame AECs: by Fact \ref{existence-extension-tame}, type-full good frames exist under natural conditions there. Moreover by \cite[Remark 3.10]{gv-superstability-v3}, if a tame AECs has a good frame, then it has a type-full one (possibly with a different class of models). We do not know if the type-full assumption is necessary; our argument uses it when dealing with minimal types (we do not know in general whether minimal types are basic; if this is the case for the frame we are working with then it is not necessary to assume that it is type-full).

What about categoricity in $\lambda$? This is assumed in order to have some starting degree of saturation (namely all the models of size $\lambda$ are limit models, see Definition \ref{limit-def}). We do not see it as a strong assumption: in applications, we will take the AEC to be a class of $\lambda$-saturated models, where this automatically holds. Still, we do not know if it is necessary.

Another natural question is whether one really needs to assume the existence of a global good frame at all. The Hart-Shelah example \cite{hs-example, bk-hs} shows that it is \emph{not true} that any AEC $\K$ categorical in $\LS (\K)$ and in a successor $\lambda > \LS (\K)$ is categorical everywhere (even if $\K$ has amalgamation). One strengthening of Theorem \ref{good-categ-transfer} assumes only that we have a good frame for \emph{saturated} models. Appendix \ref{shrinking-appendix} states this precisely and outlines a proof. Consequently, most of the result stated above hold assuming only \emph{weak tameness} instead of tameness: that is, only types over saturated models are required to be determined by their small restrictions. In Appendix \ref{more-weak-tameness} we mention in which results of the paper tameness can be replaced by weak tameness.

\subsection{Application: lowering the bounds in Shelah's transfer}

In the second part of this paper, we give several applications of Theorem \ref{good-categ-transfer} to Shelah's eventual categoricity conjecture, the central test problem in the classification theory of non-elementary classes (see the introduction of \cite{ap-universal-v9} for a history):

\begin{conjecture}[Conjecture N.4.2 in \cite{shelahaecbook}]
  An AEC that is categorical in a high-enough cardinal is categorical on a tail of cardinals.
\end{conjecture}

For an AEC $\K$ we will call \emph{Shelah's categoricity conjecture for $\K$} the statement that if $\K$ is categorical in \emph{some} $\lambda \ge \ehanf{\LS (\K)}$, then $\K$ is categorical in \emph{all} $\lambda' \ge \ehanf{\LS (\K)}$ (that is, we explicitly require the ``high-enough'' threshold to be the first Hanf number).

Shelah \cite{sh394} has proven a \emph{downward} categoricity transfer from a successor in AECs with amalgamation where the threshold is the second Hanf number. Complementing it, Grossberg and VanDieren have established an \emph{upward} transfer assuming tameness:

\begin{fact}[\cite{tamenesstwo, tamenessthree}]\label{gv-upward}
  Let $\K$ be a $\LS (\K)$-tame AEC with amalgamation and arbitrarily large models. Let $\lambda > \LS (\K)^+$ be a successor cardinal. If $\K$ is categorical in $\lambda$, then $\K$ is categorical in all $\lambda' \ge \lambda$.
\end{fact}

Grossberg and VanDieren concluded that Shelah's \emph{eventual} categoricity conjecture from a successor holds in tame AECs with amalgamation. Baldwin \cite[Problem D.1.(5)]{baldwinbook09} has asked whether the threshold in Shelah's downward transfer can be lowered to the first Hanf number. The answer is not known, but we show here that tameness is the only obstacle: assuming $\LS (\K)$-tameness, the threshold becomes the first Hanf number, and so using Fact \ref{gv-upward}, we obtain Shelah's categoricity conjecture from a successor in tame AECs with amalgamation:

\textbf{Corollary \ref{main-thm}.}
  Let $\K$ be a $\LS (\K)$-tame AEC with amalgamation and arbitrarily large models. If $\K$ is categorical in \emph{some successor} $\lambda > \LS (\K)^+$, then $\K$ is categorical in \emph{all} $\lambda' \ge \min (\lambda, \beth_{\left(2^{\LS (\K)}\right)^+})$.

This can be seen as a generalization (see \cite{tamelc-jsl}) of the corresponding result of Makkai and Shelah \cite{makkaishelah} for classes of models of an $\Ll_{\kappa, \omega}$-theory, $\kappa$ a strongly compact cardinal. It is a central open question whether tameness follows from categoricity in AECs with amalgamation (see \cite[Conjecture 1.5]{tamenessthree}).

We can use Theorem \ref{good-categ-transfer} to give alternate proofs of Shelah's downward transfer \cite{sh394} (see Corollary \ref{sh394-alternate} in the appendix) and for the Grossberg-VanDieren upward transfer (see Corollary \ref{gv-alternate} in the appendix). We also prove a local categoricity transfer that does not mention frames:

\textbf{Corollary \ref{main-cor}.}
  Let $\K$ be a $\LS (\K)$-tame AEC with amalgamation and arbitrarily large models. Let\footnote{On the case $\lambda_0 = \LS (\K)$, see Remark \ref{main-cor-rmk}.} $\LS (\K) < \lambda_0 < \lambda$. If $\lambda$ is a successor cardinal and $\K$ is categorical in both $\lambda_0$ and $\lambda$, then $\K$ is categorical in all $\lambda' \in [\lambda_0, \lambda]$

\begin{remark}
  We believe that the methods of \cite{sh394} are not sufficient to prove Corollary \ref{main-cor} (indeed, Shelah uses models of set theory to prove the transfer of ``no Vaughtian pairs'', and hence uses that the starting categoricity cardinal is above the Hanf number, see $(\ast)_9$ in the proof of \cite[Theorem II.2.7]{sh394}, or \cite[Theorem 14.12]{baldwinbook09}). However we noticed after posting a first draft of this paper that they \emph{are} enough to improve the threshold cardinal of \cite{sh394} from $\beth_{\beth_{\left(2^{\LS (\K)}\right)^+}}$ to $\beth_{\left(2^{\LS (\K)}\right)^+}$. We sketch the details in Section \ref{shelah-sec}.
\end{remark}

\subsection{Application: categoricity in a limit, using primes}

Beyond categoricity in a successor, we can appeal to Theorem \ref{good-categ-transfer} to give improvements on the threshold of our previous categoricity transfer in tame AECs with amalgamation and primes \cite{categ-primes-v3}: there the threshold was the second Hanf number (see Fact \ref{prime-fact}) and here we show that the first Hanf number suffices:

\textbf{Corollary \ref{improved-prime-categ}.}
    Let $\K$ be a $\LS (\K)$-tame AEC with amalgamation and arbitrarily large models. Assume that $\K$ has primes. If $\K$ is categorical in \emph{some} $\lambda > \LS (\K)$, then $\K$ is categorical in \emph{all} $\lambda' \ge \min (\lambda, \beth_{\left(2^{\LS (\K)}\right)^+})$.

\begin{remark}
  Compared to Fact \ref{gv-upward} and Corollary \ref{main-thm}, the case $\lambda = \LS (\K)^+$ is allowed. We can also allow $\lambda = \LS (\K)^+$ if we assume the weak generalized continuum hypothesis instead of the existence of primes, see Corollary \ref{abstract-thm-3-proof}.
\end{remark}

\subsection{Categoricity in a limit, using WGCH}

Finally, a natural question is how to deal with categoricity in a limit cardinal \emph{without} assuming the existence of prime models. In \cite[Theorem IV.7.12]{shelahaecbook}, Shelah claims assuming the weak generalized continuum hypothesis that if $\K$ is an AEC with amalgamation\footnote{Shelah only assumes some instances of amalgamation and no maximal models at specific cardinals, see the discussion in Section \ref{categ-limit-sec}.}, then\footnote{Shelah gives a stronger, erroneous statement (it contradicts Morley's categoricity theorem) but this is what his proof gives.} categoricity in some $\lambda \ge \ehanf{\aleph_{\LS (\K)^+}}$ implies categoricity in all $\lambda' \ge \ehanf{\aleph_{\LS (\K)^+}}$. Shelah's argument relies on an unpublished claim (whose proof should appear in \cite{sh842}), as well as PCF theory and long constructions of linear orders from \cite[Sections IV.5,IV.6]{shelahaecbook}. We have not fully verified it. In \cite{indep-aec-apal}, we gave a way to work around the use of PCF theory and the construction of linear orders (though still using Shelah's unpublished claim) by using the locality assumption of full tameness and shortness (a stronger assumption than tameness introduced by Will Boney in his Ph.D.\ thesis, see \cite[Definition 3.3]{tamelc-jsl}).

In Section \ref{categ-limit-sec}, we give an exposition of Shelah's proof that does not use PCF or the construction of linear orders. This uses a recent result of VanDieren and the author \cite[Corollary 7.4]{vv-symmetry-transfer-v3}, showing that a model at a high-enough categoricity cardinal must have some degree of saturation (regardless of the cofinality of the cardinal). We deduce (using the aforementioned unpublished claim of Shelah) that Shelah's eventual categoricity conjecture is consistent assuming the existence of a proper class of \emph{measurable} cardinals (this was implicit in \cite[Chapter IV]{shelahaecbook} but we give the details). Furthermore we give an explicit upper bound on the categoricity threshold (see Theorem \ref{categ-measurable}). This partially answers \cite[Question 6.14.(1)]{sh702}.

Using Theorem \ref{good-categ-transfer}, we also give an improvement on the categoricity threshold of $\ehanf{\aleph_{\LS (\K)^+}}$  if the AEC is tame:

\textbf{Corollary \ref{abstract-thm-3-proof}.}
  Assume the weak generalized continuum hypothesis and an unpublished claim of Shelah (Claim \ref{claim-xxx}). Let $\K$ be a $\LS (\K)$-tame AEC with amalgamation and arbitrarily large models. If $\K$ is categorical in \emph{some} $\lambda > \LS (\K)$, then $\K$ is categorical in \emph{all} $\lambda' \ge \min (\lambda, \beth_{\left(2^{\LS (\K)}\right)^+})$.

Moreover, we give two ZFC consequences of a lemma in Shelah's proof (which obtains weak tameness from categoricity in certain cardinals below the Hanf number): an improvement on the Hanf number for constructing good frames (Theorem \ref{good-frame-improvement}) and a nontrivial restriction on the categoricity spectrum \emph{below} the Hanf number of an AEC with amalgamation and no maximal models (Theorem \ref{categ-below-hanf}).

For clarity, we emphasize once again that Corollary \ref{abstract-thm-3-proof} is due to Shelah when the threshold is $\ehanf{\aleph_{\LS (\K)^+}}$ (and then tameness is not needed). The main contribution of Section \ref{categ-limit-sec} is a clear outline of Shelah's proof that avoids several of his harder arguments.

In conclusion, the aim of the second part of this paper is to clarify the status of Shelah's eventual categoricity conjecture by simplifying existing proofs and improving several thresholds. We give a table summarizing the known results on the conjecture in Section \ref{summary-sec}.

\subsection{Notes and acknowledgments}

After the initial circulation of this paper (in October 2015), we showed \cite{categ-saturated-v2} that if $\K$ is an AEC with amalgamation and no maximal models that is categorical in $\lambda > \LS (\K)$ then the model of categoricity $\lambda$ is always saturated. Thus several of the threshold cardinals in Sections \ref{weak-tameness-sec}, \ref{shelah-sec}, or \ref{categ-limit-sec} can be improved. In particular in  the last two entry of the first column of Table \ref{summary-table} in Section \ref{summary-sec}, the cardinal $\beth_{\beth_{H_1}}$ can be replaced by $\beth_{H_1}$.

The background required to read this paper is a solid knowledge of AECs (at minimum Baldwin's book \cite{baldwinbook09}) together with some familiarity with good frames (e.g.\ the first four sections of \cite[Chapter II]{shelahaecbook}). As mentioned before, the paper has two parts: The first gives a proof of the main Theorem (Theorem \ref{good-categ-transfer}), and the second gives applications. If one is willing to take Theorem \ref{good-categ-transfer} as a black box, the second part can be read independently from the first part. The first part relies on \cite{tame-frames-revisited-v5} and the second relies on several other recent preprints (e.g.\ \cite{ss-tame-jsl, categ-primes-v3, vv-symmetry-transfer-v3}, as well as on parts of \cite[Chapter IV]{shelahaecbook} (we only use results for which Shelah gives a full proof). We have tried to state all background facts as black boxes that can be used with little understanding of the underlying machinery.

We warn the reader: at the beginning of most sections, we state a \emph{global} hypothesis which applies to any result stated in the section.

This paper was written while working on a Ph.D.\ thesis under the direction of Rami Grossberg at Carnegie Mellon University and I would like to thank Professor Grossberg for his guidance and assistance in my research in general and in this work specifically. I thank John Baldwin and Monica VanDieren for helpful feedback on an earlier draft of this paper. I also thank Will Boney for a conversation on Shelah's omitting type theorem (see Section \ref{shelah-sec}). Finally, I thank the referee for comments that helped improve the presentation of this paper.

\part{The main theorem}

\section{Background}

We assume that the reader is familiar with the definition of an AEC and notions such as amalgamation, joint embedding, Galois types, and Ehrenfeucht-Mostowski models (see for example \cite{baldwinbook09}). The notation we use is standard and is described in details at the beginning of \cite{sv-infinitary-stability-afml}. For example, we write $\gtp (\bb / M; N)$ for the Galois type of $\bb$ over $M$, as computed in $N$. Everywhere in this paper and unless mentioned otherwise, we are working inside a fixed AEC $\K$.

\subsection{Good frames}

In \cite[Definition II.2.1]{shelahaecbook}\footnote{The definition here is simpler and more general than the original: We will \emph{not} use Shelah's axiom (B) requiring the existence of a superlimit model of size $\lambda$. Several papers (e.g.\ \cite{jrsh875}) define good frames without this assumption.}, Shelah introduces \emph{good frames}, a local notion of independence for AECs. This is the central concept of his book and has seen several other applications, such as a proof of Shelah's eventual categoricity conjecture for universal classes \cite{ap-universal-v9, categ-universal-2-v2}. A \emph{good $\lambda$-frame} is a triple $\s = (\K_\lambda, \nf, \Sbs)$ where:

\begin{enumerate}
  \item $\K$ is a nonempty AEC which has amalgamation in $\lambda$, joint embedding in $\lambda$, no maximal models in $\lambda$, and is stable in $\lambda$.
  \item For each $M \in \K_\lambda$, $\Sbs (M)$ (called the set of \emph{basic types} over $M$) is a set of nonalgebraic Galois types over $M$ satisfying (among others) the \emph{density property}: if $M \lta N$ are in $\K_\lambda$, there exists $a \in |N| \backslash |M|$ such that $\gtp (a / M; N) \in \Sbs (M)$.
  \item $\nf$ is an (abstract) independence relation on types of length one over models in $\K_\lambda$ satisfying the basic properties of first-order forking in a superstable theory: invariance, monotonicity, extension, uniqueness, transitivity, local character, and symmetry (see \cite[Definition II.2.1]{shelahaecbook}).
\end{enumerate}

As in \cite[Definition II.6.35]{shelahaecbook}, we say that a good $\lambda$-frame $\s$ is \emph{type-full} if for each $M \in \K_\lambda$, $\Sbs (M)$ consists of \emph{all} the nonalgebraic types over $M$. We focus on type-full good frames in this paper and hence just write $\s = (\K_\lambda, \nf)$. For notational simplicity, we extend forking to algebraic types by specifying that algebraic types do not fork over their domain. Given a type-full good $\mu$-frame $\s = (\K_\lambda, \nf)$ and $M_0 \lea M$ both in $\K_\lambda$, we say that a nonalgebraic type $p \in \gS (M)$ \emph{does not $\s$-fork over $M_0$} if it does not fork over $M_0$ according to the abstract independence relation $\nf$ of $\s$. When $\s$ is clear from context, we omit it and just say that \emph{$p$ does not fork over $M_0$.} We write $\K_{\s}$ for the underlying class (containing only models of size $\lambda$) of $\s$. We say that a good $\lambda$-frame $\s$ is \emph{on $\K_\lambda$} if $\K_{\s} = \K_\lambda$. We might also just say that $\s$ is \emph{on $\K$}.

We more generally look at frames where the forking relation works over larger models. For $\F = [\lambda, \theta)$ an interval with $\theta \ge \lambda$ a cardinal or $\infty$, we define a \emph{type-full good $\F$-frame} similarly to a type-full good $\lambda$-frame but require forking to be defined over models in $\K_{\F}$ (similarly, the good properties hold of the class $\K_{\F}$, e.g.\ $\K$ is stable in every $\mu \in \F$). See \cite[Definition 2.21]{ss-tame-jsl} for a precise definition. For a type-full good $\F$-frame $\s = (\K_{\F}, \nf)$ and $\K'$ a subclass of $\K_{\F}$, we define the restriction $\s \rest \K'$ of $\s$ to $\K'$ in the natural way (see \cite[Notation 3.17.(2)]{indep-aec-apal}).

At one point in the paper (Section \ref{domination-sec}) we will look at (not necessarily good) frames defined over types longer than one element. This was first defined in \cite[Definition 3.1]{tame-frames-revisited-v5} but we use \cite[Definition 3.1]{indep-aec-apal}. We require in addition that it satisfies the base monotonicity property and that the underlying class is an AEC.
  
\begin{defin}\label{pre-frame-def}
  A type-full pre-$(\le \lambda, \lambda)$-frame is a pair $\ts := (K_\lambda, \nf)$, where $\K$ is an AEC with amalgamation in $\lambda$ and $\nf$ is a relation on types of length at most $\lambda$ over models in $\K_{\lambda}$ satisfying invariance and monotonicity (including base monotonicity, see \cite[Definition 3.12.(4)]{indep-aec-apal}).

  For a type-full good $\lambda$-frame $\s$, we say that \emph{$\ts$ extends $\s$} if they have the same underlying class and forking in $\s$ and $\ts$ coincide.
\end{defin}

  \subsection{Saturated and limit models}

  We will make heavy use of limit models (see \cite{gvv-mlq} for history and motivation). Here we give a global definition, where we permit the limit model and the base to have different sizes.  

  \begin{defin}\label{limit-def}
    Let $M_0 \lea M$ be models in $\K_{\ge \LS (\K)}$. $M$ is \emph{limit over $M_0$} if there exists a limit ordinal $\delta$ and a strictly increasing continuous sequence $\seq{N_i : i \le \delta}$ such that:

  \begin{enumerate}
    \item $N_0 = M_0$.
    \item $N_\delta = M$.
    \item For all $i < \delta$, $N_{i + 1}$ is universal over $N_i$ (that is, for any $N \in \K_{\|N_i\|}$ with $N_0 \lea N$, there exists $f: N \xrightarrow[N_i]{} N_{i + 1}$).
  \end{enumerate}

  We say that $M$ is \emph{limit} if it is limit over some $M' \lea M$.
\end{defin}

\begin{defin}\label{ksat-def}
  Assume that $\K$ has amalgamation. 

  \begin{enumerate}
    \item For $\lambda > \LS (\K)$, $\Ksatp{\lambda}$ is the class of $\lambda$-saturated (in the sense of Galois types) models in $\K_{\ge \lambda}$. We order it with the strong substructure relation inherited from $\K$.
    \item We also define $\Ksatp{\LS (\K)}$ to be the class of models $M \in \K_{\ge \LS (\K)}$ such that for all $A \subseteq |M|$ with $|A| \le \LS (\K)$, there exists a limit model $M_0 \lea M$ with $M_0 \in \K_{\LS (\K)}$ and $A \subseteq |M_0|$. We order $\Ksatp{\LS (\K)}$ with the strong substructure relation inherited from $\K$.
  \end{enumerate}
\end{defin}
\begin{remark}
  If $\K$ has amalgamation and is stable in $\LS (\K)$, then $\Ksatp{\LS (\K)}_{\LS (\K)}$ is the class of limit models in $\K_{\LS (\K)}$.
\end{remark}

We will repeatedly use the uniqueness of limit models inside a good frame, first proven by Shelah in \cite[Claim II.4.8]{shelahaecbook} (see also \cite[Theorem 9.2]{ext-frame-jml}).

\begin{fact}\label{uq-limit}
  Let $\s$ be a type-full good $\lambda$-frame.
  
  Let $M_0, M_1, M_2 \in \K_\s$. 

  \begin{enumerate}
    \item If $M_1$ and $M_2$ are limit models, then $M_1 \cong M_2$.
    \item If in addition $M_1$ and $M_2$ are both limit over $M_0$, then $M_1 \cong_{M_0} M_2$.
  \end{enumerate}
\end{fact}

For global frames, we can combine this with a result of VanDieren \cite{vandieren-chainsat-apal} to obtain that limit models are saturated and closed under unions:

\begin{fact}\label{satfact}
  Let $\s$ be a type-full good $[\lambda, \theta)$-frame on the AEC $\K$. Let $\mu \in [\lambda, \theta)$.

      \begin{enumerate}
      \item $M \in \Ksatp{\mu}_\mu$ if and only if $M$ is limit.
      \item If $\mu > \lambda$, then $\Ksatp{\mu}_{[\mu, \theta)}$ is the initial segment of an AEC with Löwenheim-Skolem-Tarski number $\mu$.
      \end{enumerate}
\end{fact}
\begin{proof} \
  \begin{enumerate}
  \item This is trivial if $\lambda = \LS (\K)$, so assume that $\lambda > \LS (\K)$. If $M$ is limit, then by uniqueness of limit models, $M$ is saturated. Conversely if $M$ is saturated, then it must be unique, hence isomorphic to a limit model.
  \item By uniqueness of limit models (Fact \ref{uq-limit}) and \cite[Theorem 1]{vandieren-chainsat-apal}. Note that there a condition called $\mu$-superstability (see Definition \ref{ss assm}) rather than the existence of a good $\mu$-frame is used. However the existence of a type-full good $\mu$-frame implies $\mu$-superstability (see Fact \ref{ss-good-frame}).
  \end{enumerate}
\end{proof}

We will use Facts \ref{uq-limit} and \ref{satfact} freely.

\section{Domination and uniqueness triples}\label{domination-sec}

In this section, we assume:

\begin{hypothesis}\label{domination-hyp}
  $\s = (\K_\lambda, \nf)$ is a type-full good $\lambda$-frame.
\end{hypothesis}
\begin{remark}
  The results of this section can be adapted to non-type-full good frames, but we assume type-fullness anyway for notational convenience.
\end{remark}

Our aim (for the next sections) is to develop some orthogonality calculus as in \cite[Section III.6]{shelahaecbook}. There Shelah works in a good $\lambda$-frame that is \emph{successful} (see \cite[Definition III.1.1]{shelahaecbook}). Note that by \cite[Claim III.9.6]{shelahaecbook} such a good frame can be extended to a type-full one. Thus the framework of this section is more general (see \cite{counterexample-frame-v2} for an example of a non-successful type-full frame).

One of the main components of the definition of successful is the existence property for uniqueness triples (see \cite[Definition III.1.1]{shelahaecbook}, \cite[Definition 4.1.(5)]{jrsh875}, or here Definition \ref{weakly-succ-def}). It was shown in \cite[Lemma 11.7]{indep-aec-apal} that this property is equivalent to a version of domination assuming the existence of a global independence relation. Using an argument of Makkai and Shelah \cite[Proposition 4.22]{makkaishelah}, one can see \cite[Lemma 11.12]{indep-aec-apal} that this version of domination satisfies a natural existence property. We give a slight improvement on this result here by working only locally in $\lambda$ (i.e.\ using limit models rather than saturated ones).

Crucial in this section is the uniqueness of limit models (Fact \ref{uq-limit}). A consequence is the following conjugation property \cite[Claim 1.21]{shelahaecbook}. It is stated there for $M, N$ superlimit but the proof goes through if $M$ and $N$ are limit models.

\begin{fact}[The conjugation property]\label{conj-prop}
  If $M \lea N$ are limit models in $\K_{\lambda}$ and $p, q \in \gS (N)$ do not fork over $M$, then there exists $f: N \cong M$ so that $f (p) = p \rest M$ and $f (q) = q \rest M$.
\end{fact}

The next definition is modeled on \cite[Definition 11.5]{indep-aec-apal}:

\begin{defin}\label{dom-triple-def}
  Let $\ts$ be a pre-$(\le \lambda, \lambda)$-frame extending $\s$  (see Definition \ref{pre-frame-def}, we write $\nf$ for the nonforking relation of both $\s$ and $\ts$).
  The triple $(a, M, N)$ is a \emph{domination triple for $\ts$} if $M, N \in K_\lambda$, $M \lea N$, $a \in |N| \backslash |M|$ and for any $N' \gea N$ and $M' \lea N$ with $M \lea M'$ and $M', N' \in K_\lambda$, if $\nfs{M}{a}{M'}{N'}$, then $\nfs{M}{N}{M'}{N'}$.
\end{defin}
\begin{remark}
  In this paper, we will take $\ts$ to be 1-forking (Definition \ref{1-forking-def}), so the reader who wants a concrete example may substitute it for $\ts$ throughout this section.
\end{remark}

Domination triples are related to Shelah's uniqueness triples (see \cite[Definition III.1.1]{shelahaecbook} or \cite[Definition 4.1.(5)]{jrsh875}) by the following result (this will not be used outside of this section):

\begin{fact}[Lemma 11.7 in \cite{indep-aec-apal}]\label{uq-triple-fact}
Let $\ts$ be a pre-$(\le \lambda, \lambda)$-frame extending $\s$. If $\ts$ has the uniqueness property, then any domination triple for $\ts$ is a uniqueness triple for $\s$.
\end{fact}

We now want to show the existence property for domination triples: For any type $p \in \gS (M)$, there exists a domination triple $(a, M, N)$ with $p = \gtp (a / M; N)$. We manage to do it when $M$ is a limit model. The proof is a local version of \cite[Lemma 11.12]{indep-aec-apal} (which adapted \cite[Proposition 4.22]{makkaishelah}). We will consider the following local character properties that $\ts$ may have:

\begin{defin}\label{star-def}
  Let $\ts$ be a pre-$(\le \lambda, \lambda)$-frame extending $\s$. Let $\kappa \ge 2$ be a cardinal.

  \begin{enumerate}
    \item We say that \emph{$\ts$ satisfies local character for $(<\kappa)$-length types over $(\lambda, \lambda^+)$-limits} if whenever $\seq{M_i : i < \lambda^+}$ is increasing in $\K_{\lambda}$ with $M_{i + 1}$ universal over $M_i$ for all $i < \lambda^+$ and $p \in \gS^{<\kappa} (\bigcup_{i < \lambda^+} M_i)$, there exists $i < \lambda^+$ such that for any $j \in [i, \lambda^+)$, $p \rest M_j$ does not fork over $M_i$.
    \item We say that \emph{$\ts$ reflects down} if whenever $\seq{M_i : i < \lambda^+}$, $\seq{N_i : i < \lambda^+}$ are increasing continuous in $\K_{\lambda}$ so that for all $i < \lambda^+$, $M_i \lea N_i$, $M_{i + 1}$ is universal over $M_i$, and $N_{i + 1}$ is universal over $N_i$, then there exists $i < \lambda^+$ such that $\nfs{M_i}{N_i}{M_{i + 1}}{N_{i + 1}}$.
    \item\cite[Definition 3.12.(9)]{indep-aec-apal} For $\kappa \ge 2$, we say that \emph{$\ts$ has the left $(<\kappa)$-witness property} if for any three models $M_0 \lea M \lea N$ in $\K_{\lambda}$ and any $A \subseteq |N|$, $\nfs{M_0}{A}{M}{N}$ holds if and only if $\nfs{M_0}{A_0}{M}{N}$ for all $A_0 \subseteq A$ with $|A_0| < \kappa$.
  \end{enumerate}
\end{defin}

Note that the witness property implies some amount of local character:

\begin{lem}\label{witness-lc}
  Let $\kappa \ge 2$. Let $\ts$ be a pre-$(\le \lambda, \lambda)$-frame extending $\s$. If $\lambda = \lambda^{<\kappa}$, $\ts$ satisfies local character for $(<\kappa)$-length types over $(\lambda, \lambda^+)$-limits, and $\ts$ has the $(<\kappa)$-witness property, then $\ts$ satisfies local character for $(<\lambda^+)$-length types over $(\lambda, \lambda^+)$-limits.
\end{lem}
\begin{proof}
  Let $\seq{M_i : i < \lambda^+}$ be increasing in $\K_{\lambda}$, with $M_{i + 1}$ universal over $M_i$ for all $i < \lambda^+$. Write $M_{\lambda^+} := \bigcup_{i < \lambda^+} M_i$. Let $p \in \gS^{\alpha} (M_{\lambda^+})$ with $\alpha < \lambda^+$. Say $p = \gtp (\ba / M_{\lambda^+}; N)$. For each $I \subseteq \alpha$ with $|I| < \kappa$, we will write $p^I$ for $\gtp (\ba \rest I / M_{\lambda^+}; N)$.

  Directly from the local character assumption on $\ts$, we have that for each $I \subseteq \alpha$ with $|I| < \kappa$, there exists $i_I < \lambda^+$ so that $p^I \rest M_j$ does not fork over $M_{i_I}$ for all $j \ge i_I$.

  Now let $i := \sup_{I \subseteq \alpha, |I| < \kappa} i_k$. Since $\lambda = \lambda^{<\kappa}$, $i < \lambda^+$. Using the witness property, we get that $p \rest M_j$ does not fork over $M_i$ for all $j \ge i$, as desired.
\end{proof}

Moreover local character together with the witness property imply that $\ts$ reflects down:

\begin{lem}\label{seq-lc}
  Let $\ts$ be a pre-$(\le \lambda, \lambda)$-frame extending $\s$. If $\ts$ has local character for $(<\lambda^+)$-length types over $(\lambda, \lambda^+)$-limits and has the left $(<\kappa)$-witness property for some regular $\kappa \le \lambda$, then $\ts$ reflects down.
\end{lem}
\begin{proof}
  Fix $\seq{M_i : i < \lambda^+}$, $\seq{N_i : i < \lambda^+}$ as in the definition of reflecting down. By local character for $\ts$, for each $i < \lambda^+$, there exists $j_i < \lambda^+$ such that $\nfs{M_{j_i}}{N_i}{M_j}{N_j}$ for all $j \ge j_i$. Let $i^\ast < \lambda^+$ be such that $\cf{i^\ast} = \kappa$ and $j_i < i^\ast$ for all $i < i^\ast$. Then it is easy to check using the left $(<\kappa)$-witness property that $\nfs{M_{i^\ast}}{N_{i^\ast}}{M_j}{N_j}$ for all $j \ge i^\ast$, which is as needed.
\end{proof}

We have arrived to the existence property for domination triples. For the convenience of the reader, we restate Hypothesis \ref{domination-hyp}

\begin{thm}\label{weak-dom-existence-0}
  Let $\s$ be a type-full good $\lambda$-frame on $\K$ and let $\ts$ be a pre-$(\le \lambda, \lambda)$-frame extending $\s$. Assume that $\ts$ reflects down.

  Let $M \in K_\lambda$ be a limit model. For each nonalgebraic $p \in \gS (M)$, there exists a domination triple $(a, M, N)$ for $\ts$ such that $p = \gtp (a / M; N)$.
\end{thm}
\begin{proof}
  Assume not.

  \paragraph{\underline{Claim}} For any limit $M' \in \K_{\lambda}$ with $M' \gea M$, if $q \in \gS (M')$ is the nonforking extension of $p$, then there is no domination triple $(b, M', N')$ such that $q = \gtp (b / M'; N')$. 

  \paragraph{\underline{Proof of claim}} By the conjugation property (Fact \ref{conj-prop}), there exists $f: M' \cong M$ such that $f (q) = p$. Now use that domination triples are invariant under isomorphisms. $\dagger_{\text{Claim}}$
  
  We construct $a$, $\seq{M_i : i < \lambda^+}$, $\seq{N_i : i < \lambda^+}$ increasing continuous such that for all $i < \lambda^+$:

  \begin{enumerate}
    \item $M_0 = M$.
    \item $M_i \lea N_i$ are both in $\K_\lambda$.
    \item $M_{i + 1}$ is limit over $M_i$ and $N_{i + 1}$ is limit over $N_i$.
    \item $\gtp (a / M_i; N_i)$ is the nonforking extension of $p$. In particular, $\nfs{M_0}{a}{M_i}{N_i}$.
    \item $\nnfs{M_i}{N_i}{M_{i + 1}}{N_{i + 1}}$.
  \end{enumerate}

  This is enough, since then we get a contradiction to $\ts$ reflecting down. This is possible: If $i = 0$, let $N_0 \in \K_\lambda$ and $a \in |N_0|$ be such that $p = \gtp (a / M; N_0)$. At limits, take unions. Now assume everything up to $i$ has been constructed. By the claim, $(a, M_i, N_i)$ cannot be a domination triple. This means there exists $M_i' \gea M_i$ and $N_i' \gea N_i$ all in $\K_\lambda$ such that $\nfs{M_i}{a}{M_i'}{N_i'}$ but $\nnfs{M_i}{N_i}{M_i'}{N_i'}$. By the extension property of forking, pick $M_{i + 1} \in \K_\lambda$ limit over $M_i$ containing $M_i'$ and $N_{i + 1} \gea N_i'$ such that $N_{i + 1}$ is limit over $N_i$ and $\nfs{M_i}{a}{M_{i + 1}}{N_{i + 1}}$. 
\end{proof}

The next corollary will not be used in the rest of this paper. It improves on \cite[Theorem 11.13]{indep-aec-apal} by working exclusively in $\lambda$ (so there is no need to assume the existence of a good frame below $\lambda$). 

\begin{cor}\label{weakly-succ-cor}
  Assume that $\K$ is categorical in $\lambda$. If there exists a pre-$(\le \lambda, \lambda)$-frame $\ts$ extending $\s$, reflecting down, and satisfying uniqueness, then $\s$ has the existence property for uniqueness triples (i.e.\ it is weakly successful, see \cite[Definition III.1.1]{shelahaecbook} or Definition \ref{weakly-succ-def}).
\end{cor}
\begin{proof}
  By Theorem \ref{weak-dom-existence-0}, $\ts$ has the existence property for domination triples. By Fact \ref{uq-triple-fact}, any domination triple is a uniqueness triple.
\end{proof}

\begin{question}
  Is the converse true? Namely if $\K$ is categorical in $\lambda$ and $\s$ is a weakly successful good $\lambda$-frame on $\K$, does there exist a pre-$(\le \lambda, \lambda)$-frame $\ts$ that extends $\s$, satisfies uniqueness, and reflects down?  
\end{question}

It is known (see \cite[Section II.6]{shelahaecbook}  and \cite[Section 12]{indep-aec-apal}) that weakly successful good $\lambda$-frame can be extended to pre-$(\le \lambda, \lambda)$-frame with several good properties, including uniqueness, but it is not clear that the pre-frame reflects down.

We finish with a slight improvement on the construction of a weakly successful good frame from \cite[Appendix A]{ap-universal-v9} (which improved the threshold cardinals of \cite{indep-aec-apal}). Since the result is not needed for the rest of the paper, we only sketch the proof and quote freely. At this point, we drop Hypothesis \ref{domination-hyp}. 

\begin{cor}
  Let $\K$ be an AEC and let $\lambda > \LS (\K)$. Let $\kappa \le \LS (\K)$ be an infinite cardinal. Assume that $\LS (\K) = \LS (\K)^{<\kappa}$ and $\lambda = \lambda^{<\kappa}$. Assume that $\K$ is $\LS (\K)$-tame for all types of length less than $\kappa$ over saturated models of size $\lambda$, and $\K$ is $(<\kappa)$-type short for all types of length at most $\lambda$ over saturated models of size $\lambda$ (see \cite[Definition 3.3]{tamelc-jsl} and \cite[Definition 2.22]{sv-infinitary-stability-afml}).

  If $\K_{[\LS (\K), \lambda]}$ has amalgamation, $\K$ is stable in $\LS (\K)$, and $\s$ is a good $\lambda$-frame on $\Ksatp{\lambda}_\lambda$, then $\s$ is weakly successful.
\end{cor}
\begin{proof}[Proof sketch]
  We use Corollary \ref{weakly-succ-cor} with the pre-$(\le \lambda, \lambda)$-frame $\ts$ on $\K_\lambda$ where $\ts$-forking is defined as follows: For $M_0 \lea M$ both in $\K_\lambda$, $p \in \gS^{\le \lambda} (M)$ does not $\ts$-fork over $M_0$ if for every $I \subseteq \ell (p)$ with $|I| < \kappa$, there exists $M_0' \in \K_{\LS (\K)}$ with $M_0' \lea M_0$ such that $p^I$ does not $\LS (\K)$-split over $M_0'$. We want to show that $\ts$ extends $\s$, $\ts$ has uniqueness, and $\ts$ reflects down.

  Following the proof of \cite[Lemma A.14]{ap-universal-v9}, we get that $\ts$ extends $\s$, has local character for $(<\lambda^+)$-length types over $(\lambda, \lambda^+)$-limits, and has uniqueness. Also, $\ts$ clearly has the left $(<\kappa^+)$-witness property and by assumption $\kappa^+ \le \LS (\K)^+ \le \lambda$. Thus by Lemma \ref{seq-lc} (where $\kappa$ there is $\kappa^+$ here), $\ts$ reflects down, as desired.
\end{proof}

\section{Local orthogonality}

\begin{hypothesis}
  $\s = (\K_\lambda, \nf)$ is a type-full good $\lambda$-frame.
\end{hypothesis}

The next definition is what we will use for the $\ts$ of the previous section. One can see it as a replacement for a notion of forking for types over models, when such a notion is not available. It already plays a role in \cite{makkaishelah} (see Lemma 4.17 there) and \cite{bv-sat-v3} (see Definition 3.10 there). A similar notion is called ``smooth independence'' in \cite{viza}.

\begin{defin}\label{1-forking-def}
  For $M \in \K_{\lambda}$ and $p \in \gS^{<\infty} (M)$, we say that $p$ \emph{does not 1-$\s$-fork} over $M_0$ if $M_0 \lea M$ and for $I \subseteq \ell (p)$ with $|I| = 1$, we have that $p^I$ does not $\s$-fork over $M_0$. We see 1-forking as inducing a pre-$(\le \lambda, \lambda)$-frame (Definition \ref{pre-frame-def}).
\end{defin}

\begin{notation}
  We write $\nfs{M_0}{[A]^1}{M}{N}$ if for some (any) enumeration $\ba$ of $A$, $\gtp (\ba / M; N)$ does not $1$-$\s$-fork over $M_0$. That is, $\nfs{M_0}{a}{M}{N}$ for all $a \in A$.
\end{notation}
\begin{remark}[Disjointness]\label{disj-rmk}
  Because nonforking extensions of nonalgebraic types are nonalgebraic, if $\nfs{M_0}{[A]^1}{M}{N}$, then $|M| \cap A \subseteq |M_0|$.
\end{remark}

\begin{defin}
  $(a, M, N)$ is a \emph{weak domination triple in $\s$} if it is a domination triple (Definition \ref{dom-triple-def}) for the pre-frame induced by 1-forking. That is, $M, N \in K_\lambda$, $M \lea N$, $a \in |N| \backslash |M|$ and for any $N' \gea N$ and $M' \lea N$ with $M \lea M'$ and $M', N' \in K_\lambda$, if $\nfs{M}{a}{M'}{N'}$, then $\nfs{M}{[N]^1}{M'}{N'}$.
\end{defin}

From the results of the previous section, we deduce the existence property for weak domination triples:

\begin{lem}\label{lc-lem}
1-forking reflects down (see Definition \ref{star-def}).
\end{lem}

\begin{proof}
  It is clear that 1-forking has the left $(<2)$-witness property (Definition \ref{star-def}) so by Lemmas \ref{witness-lc} and \ref{seq-lc}, it is enough to show that 1-forking has local character for $(<2)$-length types over $(\lambda, \lambda^+)$-limits. This follows from the local character property of $\s$ (see \cite[Claim II.2.11]{shelahaecbook}). In details:
  
  \paragraph{\underline{Claim}} Let $\seq{M_i : i < \lambda^+}$ be increasing continuous in $\K_\lambda$. Let $p \in \gS (\bigcup_{i < \lambda^+})$. There exists $i < \lambda^+$ so that $p \rest M_j$ does not fork over $M_i$ for all $j \ge i$.

  \paragraph{\underline{Proof of claim}} By local character (in $\s$), for each limit $j < \lambda^+$, there exists $\gamma_j < j$ so that $p \rest M_j$ does not fork over $M_{\gamma_j}$. By Fodor's lemma, there exists $i < \lambda^+$ such that for unboundedly many $j < \lambda^+$, $\gamma_j = i$. By monotonicity of forking, $i$ is as desired.
\end{proof}

\begin{thm}\label{weak-dom-existence}
  Let $M \in K_\lambda$ be a limit model. For each nonalgebraic $p \in \gS (M)$, there exists a weak domination triple $(a, M, N)$ such that $p = \gtp (a / M; N)$.
\end{thm}
\begin{proof}
  By Lemma \ref{lc-lem} and Theorem \ref{weak-dom-existence-0}.
\end{proof}

We now give a definition of orthogonality in terms of independent sequences. 

\begin{defin}[Independent sequence, III.5.2 in \cite{shelahaecbook}]
Let $\alpha$ be an ordinal.

  \begin{enumerate}
  \item $\seq{a_i:i < \alpha} \smallfrown \seq{M_i : i \leq \alpha}$ is said to be \emph{independent in $(M, M', N)$} when:

    \begin{enumerate}
    \item $(M_i)_{i \leq \alpha}$ is increasing continuous in $\K_\lambda$.
    \item $M \lea M' \lea M_0$ and $M, M' \in \K_\lambda$.
    \item $M_\alpha \lea N$ and $N \in \K_\lambda$.
    \item For every $i < \alpha$ , $\nfs{M}{a_i}{M_i}{M_{i+1}}$.
    \end{enumerate}

    $\seq{a_i : i < \alpha} \smallfrown \seq{M_i: i \leq \alpha}$ is said to be \emph{independent over $M$} when it is independent in $(M, M_0, M_\alpha)$.
  \item $\ba := \seq{a_i : i < \alpha}$ is said to be \emph{independent in $(M, M_0, N)$} when $M \lea M_0 \lea N$, $\bar{a} \in \fct{\alpha}{|N|}$, and for some $\seq{M_i : i \leq \alpha}$ and a model $N^+$ such that $M_{\alpha} \lea N^+$, $N \lea N^+$, and $\seq{a_i : i < \alpha} \smallfrown \seq{M_i : i \leq \alpha}$ is independent over $M$. When $M = M_0$, we omit it and just say that $\ba$ is independent in $(M, N)$.
%\footnotei{SV: There was previously a part 3 with the definition of an independent set, but I moved this to the section on the study of symmetry since we do not need it before.}
  %% \item A set $I$ is said to be \emph{independent in $(M, M_0, N)$} if some enumeration of $I$ is independent in $(M, M_0, N)$. $I$ is said to be \emph{finitely independent in $(M, M_0, N)$} if all its finite subsets are independent in $(M, M_0, N)$.
  \end{enumerate}
\end{defin}
\begin{remark}
We will use the definition above when $\alpha = 2$. In this case, we have that $\seq{ab}$ is independent in $(M, N)$ if and only if $\nfs{M}{a}{b}{N}$ (technically, the right hand side of the $\nf$ relation must be a model but we can remedy this by extending the nonforking relation in the natural way, as in the definition of the minimal closure in \cite[Definition 3.4]{bgkv-apal}). 
\end{remark}

\begin{defin}\label{perp-def}
  Let $M \in \K_\lambda$ and let $p, q \in \gS (M)$ be nonalgebraic. We say that $p$ is \emph{weakly orthogonal to $q$} and write $p \wkperpp{\s} q$ (or just $p \wkperp q$ if $\s$ is clear from context) if for all $N \in \K_\lambda$ with $N \gea M$ and all $a, b \in |N|$ such that $\gtp (a / M; N) = p$ and $\gtp (b / M; N) = q$, $\seq{a b}$ is independent in $(M, N)$.

  We say that \emph{$p$ is orthogonal to $q$} (written $p \perpp{}{\s} q$, or just $p \perp q$ if $\s$ is clear from context) if for every $N \in \K_\lambda$ with $N \gea M$, $p' \wkperp q'$, where $p'$, $q'$ are the nonforking extensions to $N$ of $p$ and $q$ respectively.
\end{defin}
\begin{remark}
  Definition \ref{perp-def} is equivalent to Shelah's (\cite[Definition III.6.2]{shelahaecbook}), see \cite[Claim III.6.4.(2)]{shelahaecbook} assuming that $\s$ is \emph{successful}. By a similar proof (and assuming that $\K_{\s}$ has primes), it is also equivalent to the definition in terms of primes in \cite[Definition 2.2]{categ-primes-v3}.
\end{remark}

We will use the following consequence of symmetry:

\begin{fact}[Theorem 4.2 in \cite{jasi}]\label{sym-indep}
  For any $M_0 \lea M \lea N$ all in $\K_\lambda$, if $a, b \in |N| \backslash |M_0|$, then $\seq{ab}$ is independent in $(M_0, M, N)$ if and only if $\seq{ba}$ is independent in $(M_0, M, N)$.
\end{fact}

\begin{lem}\label{perp-fund}
  Let $M \in \K_\lambda$. Let $p, q \in \gS (M)$ be nonalgebraic. 

  \begin{enumerate}
  \item\label{perp-fund-1} If $M$ is limit, then $p \perp q$ if and only if $p \wkperp q$.
  \item\label{perp-fund-2} $p \wkperp q$ if and only if $q \wkperp p$.
  \item\label{perp-fund-3} If $p \wkperp q$, then whenever $(a, M, N)$ is a weak domination triple representing $q$, $p$ is omitted in $N$. In particular, if $M$ is limit, there exists $N \in \K_\lambda$ with $M \lta N$ so that $N$ realizes $q$ but $p$ is omitted in $N$.
  \end{enumerate}
\end{lem}
\begin{proof} \
  \begin{enumerate}
    \item By the conjugation property (Fact \ref{conj-prop}). See the proof of \cite[Lemma 2.6]{categ-primes-v3}.
    \item By Fact \ref{sym-indep}.
    \item Let $N' \in \K_\lambda$ be such that $N \lea N'$ and let $b \in |N'|$ realize $p$. We have that $\seq{ab}$ is independent in $(M, N')$. Therefore there exists $M', N'' \in \K_\lambda$ so that $N \lea N''$, $M \lea M' \lea N''$, $b \in |M'|$, and $\nfs{M}{a}{M'}{N''}$. By domination, $\nfs{M}{[N]^1}{M'}{N''}$, so by disjointness (Remark \ref{disj-rmk}), $b \notin |N|$. The last sentence follows from the existence property for weak domination triple (Theorem \ref{weak-dom-existence}).
  \end{enumerate}
\end{proof}

\section{Unidimensionality}

\begin{hypothesis}\label{unidim-hyp}
  $\s = (\K_\lambda, \nf)$ is a type-full good $\lambda$-frame on $\K$ and $\K$ is categorical in $\lambda$.
\end{hypothesis}

In this section we give a definition of unidimensionality similar to the ones in \cite[Definition V.2.2]{shelahfobook} or \cite[Section III.2]{shelahaecbook}. We show that $\s$ is unidimensional if and only if $\K$ is categorical in $\lambda^+$ (this uses categoricity in $\lambda$). In the next section, we will show how to transfer unidimensionality across cardinals, hence getting the promised categoricity transfer. In \cite[Section III.2]{shelahaecbook}, Shelah gives several different definitions of unidimensionality and also shows (see \cite[III.2.3, III.2.9]{shelahaecbook}) that the so-called ``weak-unidimensionality'' is equivalent to categoricity in $\lambda^+$ (hence our definition is equivalent to Shelah's weak unidimensionality) but it is unclear how to transfer weak-unidimensionality across cardinals without assuming that the frame is successful.

Note that the hypothesis of categoricity in $\lambda$ implies that the model of size $\lambda$ is limit, hence weak orthogonality and orthogonality coincide, see Lemma \ref{perp-fund}.

Rather than defining what it means to be unidimensional, we find it clearer to define what it means to \emph{not} be unidimensional:

\begin{defin}\label{unidim-def}
  $\s$ is \emph{unidimensional} if the following is \emph{false}: for every $M \in \K_\lambda$ and every nonalgebraic $p \in \gS (M)$, there exists $M' \in \K_{\lambda}$ with $M' \gea M$ and nonalgebraic $p', q \in \gS (M')$ so that $p'$ extends $p$ and $p' \perp q$.
\end{defin}

We first give an equivalent definition using minimal types:

\begin{defin}
  For $M \in \K_\lambda$, a type $p \in \gS (M)$ is \emph{minimal} if for every $M' \in \K_\lambda$ with $M \lea M'$, $p$ has a unique nonalgebraic extension to $\gS (M')$.
\end{defin}
\begin{remark}
  Since we are working inside a good frame, any nonalgebraic type will have at least one nonalgebraic extension (the nonforking one). The nontrivial part of the definition is its uniqueness.
\end{remark}
\begin{remark}\label{min-fork}
  If $M \lea N$ are both in $\K_\lambda$ and $p \in \gS (N)$ is nonalgebraic such that $p \rest M$ is minimal, then $p$ does not fork over $M$ (because the nonforking extension of $p \rest M$ has to be $p$).
\end{remark}

By the proof of $(\ast)_5$ in \cite[Theorem II.2.7]{sh394}:

\begin{fact}[Density of minimal types]
  For any $M \in \K_\lambda$ and nonalgebraic $p \in \gS (M)$, there exists $M' \in \K_\lambda$ and $p' \in \K_\lambda$ such that $M \lea M'$, $p'$ extends $p$, and $p'$ is minimal.
\end{fact}

\begin{lem}\label{technical-multidim-equiv}
  The following are equivalent:

  \begin{enumerate}
    \item\label{equiv-1} $\s$ is not unidimensional.
    \item\label{equiv-2} For every $M \in \K_{\lambda}$ and every minimal $p \in \gS (M)$, there exists $M' \in \K_\lambda$ with $M \lea M'$ and $p', q \in \gS (M')$ nonalgebraic so that $p'$ extends $p$ and $p' \perp q$.
    \item\label{equiv-3} For every $M \in \K_{\lambda}$ and every minimal $p \in \gS (M)$, there exists a nonalgebraic $q \in \gS (M)$ with $p \perp q$.
  \end{enumerate}
\end{lem}
\begin{proof}
  (\ref{equiv-1}) implies (\ref{equiv-2}) because (\ref{equiv-2}) is a special case of (\ref{equiv-1}). Conversely, (\ref{equiv-2}) implies (\ref{equiv-1}): given $M \in \K_\lambda$ and $p \in \gS (M)$, first use density of minimal types to extend $p$ to a minimal $p' \in \gS (M')$ (so $M' \in \K_\lambda$ $M \lea M'$). Then apply (\ref{equiv-2}).

  Also, if (\ref{equiv-3}) holds, then (\ref{equiv-2}) holds with $M = M'$. Conversely, assume that (\ref{equiv-2}) holds. Let $p \in \gS (M)$ be minimal and let $p', q, M'$ witness (\ref{equiv-2}), i.e.\ $p', q \in \gS (M')$, $p'$ extends $p$ and $p' \perp q$. By Remark \ref{min-fork}, $p'$ does not fork over $M$. By the conjugation property (Fact \ref{conj-prop}), there exists $f: M' \cong M$ so that $f (p') = p$. Thus $p \perp f (q)$, hence (\ref{equiv-2}) holds.
\end{proof}

We use the characterization to show that unidimensionality implies categoricity in $\lambda^+$. This is similar to \cite[Proposition 4.25]{makkaishelah} but the proof is slightly more involved since our definition of unidimensionality is weaker. We start with a version of density of minimal types inside a fixed model. We will use the following fact, whose proof is a straightforward direct limit argument: 

\begin{fact}[Claim 0.32.(1) in \cite{sh576}]\label{limit-type}
  Let $\seq{M_i : i \le \omega}$ be an increasing continuous chain in $\K_\lambda$ and for each $i < \omega$, let $p_i \in \gS (M_i)$ be such that $j < i$ implies $p_i \rest M_j = p_j$. Then there exists $p \in \gS (M_\omega)$ so that $p \rest M_i = p_i$ for all $i < \omega$.
\end{fact}

\begin{lem}\label{minimal-ext}
  Let $M_0 \lea M$ with $M_0 \in \K_\lambda$ and $M \in \K_{> \lambda}$. Let $p \in \gS (M_0)$. Then there exists $M_1 \in \K_\lambda$ with $M_0 \lea M_1 \lea M$ and $q \in \gS (M_1)$ so that $q$ extends $p$ and for all $M' \in \K_\lambda$ with $M_1 \lea M' \lea M$, any extension of $q$ to $\gS (M')$ does not fork over $M_1$.
\end{lem}
\begin{proof}
  Suppose not. Build $\seq{N_i : i < \omega}$ increasing in $\K_\lambda$ and $\seq{q_i : i < \omega}$ such that for all $i < \omega$:

  \begin{enumerate}
    \item $N_0 = M_0$, $q_0 = p$.
    \item $N_i \lea M$.
    \item $q_i \in \gS (N_i)$ and $q_{i + 1}$ extends $q_i$.
    \item $q_{i + 1}$ forks over $N_i$.
  \end{enumerate}

  This is possible since we assumed that the lemma failed. This is enough: let $N_\omega := \bigcup_{i < \omega} N_i$. Let $q \in \gS (N_\omega)$ extend each $q_i$ (exists by Fact \ref{limit-type}). By local character, there exists $i < \omega$ such that $q$ does not fork over $N_i$, so $q \rest N_{i + 1} = q_{i + 1}$ does not fork over $N_i$, contradiction.
\end{proof}

\begin{lem}\label{unidim-categ-0}
  If $\s$ is unidimensional, then $\K$ is categorical in $\lambda^+$.
\end{lem}
\begin{proof}
  Assume that $\K$ is \emph{not} categorical in $\lambda^+$. We show that (\ref{equiv-2}) of Lemma \ref{technical-multidim-equiv} holds so $\s$ is not unidimensional. Let $M_0 \in \K_\lambda$ and let $p \in \gS (M_0)$ be minimal. We consider two cases:

  \paragraph{\underline{Case 1}} There exists $M \in \K_{\lambda^+}$, $M_1 \in \K_{\lambda}$ with $M_0 \lea M_1 \lea M$ and an extension $p' \in \gS (M_1)$ of $p$ so that $p'$ is omitted in $M$.

  Let $c \in |M| \backslash |M_1|$. Fix $M' \lea M$ in $\K_\lambda$ containing $c$ so that $M_1 \lea M'$ and let $q := \gtp (c / M_1; M')$. We claim that $q \wkperp p'$ (and so by Lemma \ref{perp-fund}, $p' \perp q$, as needed). Let $N \in \K_\lambda$ be such that $N \gea M_1$ and let $a, b \in |N|$ be such that $p' = \gtp (b / M_1; N)$, $q = \gtp (a / M_1; N)$. We want to see that $\seq{ba}$ is independent in $(M_1, N)$. We have that $\gtp (a / M_1; N) = \gtp (c / M_1; M')$, so let $N' \in \K_\lambda$ with $M' \lea N'$ and $f: N \xrightarrow[M_1]{} N'$ witness it, i.e.\ $f (a) = c$. Let $b' := f (b)$. We have that $\gtp (b' / M'; N')$ extends $p'$, and $b' \notin |M'|$ since $p'$ is omitted in $M$, hence by minimality $\gtp (b' /M'; N')$ does not fork over $M_1$. In particular, $\seq{c b'}$ is independent in $(M_1, N')$. By invariance and monotonicity, $\seq{b a}$ is independent in $(M_1, N)$.

  \paragraph{\underline{Case 2}} Not Case 1: For every $M \in \K_{\lambda^+}$, every $M_1 \in \K_{\lambda}$ with $M_0 \lea M_1 \lea M$, every extension $p' \in \gS (M_1)$ of $p$ is realized in $M$.

  By categoricity in $\lambda$ and non-categoricity in $\lambda^+$, we can find $M \in \K_{\lambda^+}$ with $M_0 \lea M$ and $q_0 \in \gS (M_0)$ omitted in $M$. Let $M_1 \in \K_\lambda$, $M_0 \lea M_1 \lea M$ and $q \in \gS (M_1)$ extend $q_0$ so that any extension of $q$ to a model $M' \lea M$ in $\K_\lambda$ does not fork over $M_1$ (this exists by Lemma \ref{minimal-ext}). Let $p' \in \gS (M_1)$ be a nonalgebraic extension of $p$. By assumption, $p$ is realized by some $c \in |M|$. Now by the same argument as above (reversing the roles of $p'$ and $q$), $p' \wkperp q$, hence $p' \perp q$, as desired.
\end{proof}
\begin{remark}
  In fact, the second case cannot happen. Otherwise, we could use the conjugation property to show that $\K$ has no $(p, \lambda)$-Vaughtian pair (in the sense of \cite[Definition 3.1]{tamenesstwo}) and apply \cite[Theorem 4.1]{tamenesstwo} to get that $\K$ is categorical in $\lambda^+$ in that case. Since the proof of case 2 is shorter, we prefer to let it stand.
\end{remark}

For the converse of Lemma \ref{unidim-categ-0}, we will use:

\begin{fact}[Theorem 6.1 in \cite{tamenessthree}]\label{no-vp-fact}
  Assume that $\K$ is categorical in $\lambda^+$. Then there exists $M \in \K_\lambda$ and a minimal type $p \in \gS (M)$ which is realized in every $N \in \K_\lambda$ with $M \lta N$.
\end{fact}
\begin{remark}
  The proof of Fact \ref{no-vp-fact} uses categoricity in $\lambda$ in a strong way (it uses that the union of an increasing chain of limit models is limit).
\end{remark}

\begin{lem}\label{unidim-categ-1}
  If $\K$ is categorical $\lambda^+$, then $\s$ is unidimensional.
\end{lem}
\begin{proof}
  By Fact \ref{no-vp-fact}, there exists $M \in \K_{\lambda}$ and a minimal $p \in \gS (M)$ so that $p$ is realized in every $N \gta M$. Now assume for a contradiction that $\K$ is \emph{not} unidimensional. Then by Lemma \ref{technical-multidim-equiv}, there exists a nonalgebraic $q \in \gS (M)$ such that $p \perp q$. By Lemma \ref{perp-fund}.(\ref{perp-fund-3}) (note that $M$ is limit by categoricity in $\lambda$), there exists $N \in \K_\lambda$ with $N \gta M$ so that $p$ is omitted in $N$, a contradiction to the choice of $p$.
\end{proof}

We have arrived to the main result of this section. For the convenience of the reader, we repeat Hypothesis \ref{unidim-hyp}.

\begin{thm}\label{unidim-categ}
  Let $\s$ be a type-full good $\lambda$-frame on $\K$. Assume that $\K$ is categorical in $\lambda$. Then $\s$ is unidimensional if and only if $\K$ is categorical in $\lambda^+$.
\end{thm}
\begin{proof}
  By Lemmas \ref{unidim-categ-0} and \ref{unidim-categ-1}.
\end{proof}

\section{Global orthogonality}

\begin{hypothesis}\label{sec-5-hyp} \
  \begin{enumerate}
    \item $\K$ is an AEC.
    \item $\theta > \LS (\K)$ is a cardinal or $\infty$. We set $\F := [\LS (\K), \theta)$.
    \item $\s = (\K_\F, \nf)$ is a type-full good $\F$-frame.
  \end{enumerate}
\end{hypothesis}

We start developing the theory of orthogonality and unidimensionality in a more global context (with no real loss, the reader can think of $\theta = \infty$ as being the main case). The main problem is to show that for $M$ sufficiently saturated, if $p, q \in \gS (M)$ do not fork over $M_0$, then $p \perp q$ if and only if $p \rest M_0 \perp q \rest M_0$. This can be done with the conjugation property in case $\|M_0\| = \|M\|$ but in general one needs to use more tools from the study of independent sequences. We will use Fact \ref{satfact} without further mention. We will also use a few facts about independent sequences:

\begin{fact}[Corollary 6.10 in \cite{tame-frames-revisited-v5}]\label{indep-facts}
  Independent sequences of length two satisfy the axioms of a good $\F$-frame. For example:
  \begin{enumerate}
    \item Monotonicity: If $\seq{ab}$ is independent in $(M_0, M, N)$ and $M_0 \lea M_0' \lea M' \lea M \lea N \lea N'$, then $\seq{ab}$ is independent in $(M_0', M', N')$.
    \item Continuity: If $\seq{M_i : i \le \delta}$ is increasing continuous, $M_\delta \lea N$, and $\seq{a b}$ is independent in $(M_0, M_i, N)$ for all $i < \delta$, then $\seq{ab}$ is independent in $(M_0, M_\delta, N)$.
  \end{enumerate}
\end{fact}
\begin{remark}
  Inside which frame do we work in when we say that $\seq{ab}$ is independent, the global frame $\s$ or its restriction to a single cardinal? By monotonicity, the answer does not matter, i.e.\ the independent sequences are the same either way. Similarly, if $\lambda > \LS (\K)$ and $M_0, M, N \in \Ksatp{\lambda}_{[\lambda, \theta)}$, then $\seq{ab}$ is independent in $(M_0, M, N)$ with respect to $\s$ if and only if it is independent in $(M_0, N)$ with respect to $\s \rest \Ksatp{\lambda}_{[\lambda, \theta)}$ (i.e.\ we can require the models witnessing the independence to be saturated). This is a simple consequence of the extension property.
\end{remark}

We now define global orthogonality. 

\begin{defin}
  Let $M \in \K_{\F}$. For $p, q \in \gS (M)$ nonalgebraic, we write $p \perp q$ for $p \perpp{}{\s \rest \K_{\|M\|}} q$, and $p \wkperp q$ for $p \wkperpp{\s \rest \K_{\|M\|}} q$ (recall Definition \ref{perp-def}).
\end{defin}

Note that a priori we need not have that if $p \perp q$ and $p', q'$ are nonforking extensions of $p$ and $q$ to big models, then $p' \perp q'$. This will be proven first (Lemma \ref{perp-nf-1}).

\begin{lem}\label{wkperp-cont}
  Let $\delta$ be a limit ordinal. Let $\seq{M_i : i \le \delta}$ be increasing continuous in $\K_{\F}$. Let $p, q \in \gS (M_\delta)$ be nonalgebraic and assume that $p \rest M_i \wkperp q \rest M_i$ for all $i < \delta$. Then $p \wkperp q$.
\end{lem}
\begin{proof}
  By the continuity property of independent sequences (Fact \ref{indep-facts}).
\end{proof}

The difference between the next lemma and the previous one is the use of $\perp$ instead of $\wkperp$.

\begin{lem}\label{perp-cont}
  Let $\delta$ be a limit ordinal. Let $\seq{M_i : i \le \delta}$ be increasing continuous in $\K_{\F}$. Let $p, q \in \gS (M_\delta)$ be nonalgebraic and assume that $p \rest M_i \perp q \rest M_i$ for all $i < \delta$. Then $p \perp q$.
\end{lem}
\begin{proof}
  By local character, there exists $i < \delta$ so that both $p$ and $q$ do not fork over $M_i$. Without loss of generality, $i = 0$. Let $\lambda := \|M_\delta\|$. If there exists $i < \delta$ so that $\lambda = \|M_i\|$, then the result follows from the definition of orthogonality. So assume that $\|M_i\| < \lambda$ for all $i < \delta$. Let $M' \in \K_\lambda$ be such that $M_\delta \lea M'$ and let $p', q'$ be the nonforking extensions to $M'$ of $p$, $q$ respectively. We want to see that $p' \wkperp q'$. Let $\seq{M_i' : i \le \delta}$ be an increasing continuous resolution of $M'$ such that $M_i \lea M_i'$ and $\|M_i'\| = \|M_i\|$ for all $i < \delta$. We know that $p' \rest M_i'$ does not fork over $M_0$, hence over $M_i$ and similarly $q' \rest M_i'$ does not fork over $M_i$. Therefore by definition of orthogonality, $p' \rest M_i' \wkperp q' \rest M_i'$. By Lemma \ref{wkperp-cont}, $p' \wkperp q'$.
\end{proof}

\begin{lem}\label{perp-nf-1}
  Let $M_0 \lea M$ be both in $\K_{\F}$. Let $p, q \in \gS (M)$ be nonalgebraic so that both do not fork over $M_0$. If $p \rest M_0 \perp q \rest M_0$, then $p \perp q$.
\end{lem}
\begin{proof}
  Let $\delta := \cf{\|M\|}$. Build $\seq{N_i : i \le \delta}$ increasing continuous such that $N_0 = M_0$, $N_\delta = M$, and $p \rest N_i \perp q \rest N_i$ for all $i \le \delta$. This is easy: at successor steps, we require $\|N_i\| = \|N_{i + 1}\|$ and use the definition of orthogonality. At limit steps, we use Lemma \ref{perp-cont}. Then $p \rest N_\delta \perp q \rest N_\delta$, but $N_\delta = M$ so $p \perp q$.
\end{proof}

\begin{question}\label{perp-nf-question}
Is the converse true? That is if $M_0 \lea M$ are in $\K_{\F}$, $p, q \in \gS (M)$ do not fork over $M_0$ and $p \perp q$, do we have that $p \rest M_0 \perp q \rest M_0$?
\end{question}

An answer to this question would be useful in order to transfer unidimensionality up in a more conceptual way than below. With a very mild additional hypothesis, we give a positive answer in Theorem \ref{perp-global} of the appendix but this is not needed for the rest of the paper.

We now go back to studying unidimensionality. We give a global definition:

\begin{defin}
  For $\lambda \in \F$, we say that $\s$ is \emph{$\lambda$-unidimensional} if the following is \emph{false}: for every limit $M \in \K_\lambda$ and every nonalgebraic $p \in \gS (M)$, there exists a limit $M' \in \K_\lambda$ with $M \lea M'$ and $p', q \in \gS (M')$ so that $p'$ extends $p$ and $p' \perp q$.
\end{defin}
\begin{remark}
  When $\lambda > \LS (\K)$, $\s$ is $\lambda$-unidimensional if and only if $\s \rest \Ksatp{\lambda}_\lambda$ is unidimensional (see Definition \ref{unidim-def}). If $\K$ is categorical in $\LS (\K)$, this also holds when $\lambda = \LS (\K)$ (if $\K$ is not categorical in $\LS (\K)$, we do not know that $\Ksatp{\LS (\K)}$ is an AEC).
\end{remark}

Our next goal is to prove (assuming categoricity in $\LS (\K)$) that $\lambda$-unidimensionality is equivalent to $\mu$-unidimensionality for every $\lambda, \mu \in \F$. We will use another characterization of $\lambda$-unidimensionality when $\lambda > \LS (\K)$. In that case, it is enough to check failure of unidimensionality with a single minimal type.

\begin{lem}\label{min-unidim}
  Let $\lambda > \LS (\K)$ be in $\F$. The following are equivalent:
  
  \begin{enumerate}
    \item\label{min-unidim-1} $\s$ is \emph{not} $\lambda$-unidimensional.
    \item\label{min-unidim-2} There exists a saturated $M \in \K_\lambda$ and nonalgebraic types $p, q \in \gS (M)$ such that $p$ is minimal and $p \perp q$.
  \end{enumerate}
\end{lem}
\begin{proof} \
  \begin{itemize}
    \item (\ref{min-unidim-1}) implies (\ref{min-unidim-2}): Assume that $\s$ is not $\lambda$-unidimensional. Let $M \in \Ksatp{\lambda}_\lambda$ and let $p \in \gS (M)$ be minimal (exists by density of minimal types and uniqueness of saturated models). By Lemma \ref{technical-multidim-equiv}, there exists $q \in \gS (M)$ so that $p \perp q$, as desired.
    \item (\ref{min-unidim-2}) implies (\ref{min-unidim-1}): Let $M \in \Ksatp{\lambda}_\lambda$ and let $p, q \in \gS (M)$ be nonalgebraic so that $p$ is minimal and $p \perp q$. We show that $\Ksatp{\lambda}$ is \emph{not} categorical in $\lambda^+$, which is enough by Theorem \ref{unidim-categ}. Fix $N \in \K_{\LS (\K)}$ with $N \lea M$ so that $p$ does not fork over $N$. Build a strictly increasing continuous chain $\seq{M_i : i \le \lambda^+}$ such that for all $i < \lambda^+$:

  \begin{enumerate}
    \item $M_i \in \Ksatp{\lambda}_\lambda$.
    \item $M_0 = M$.
    \item $p$ is omitted in $M_i$.
  \end{enumerate}

  This is enough, since then $p$ is omitted in $M_{\lambda^+}$ so $M_{\lambda^+} \in \K_{\lambda^+}$ cannot be saturated. This is possible: at limits we take unions and for $i = 0$ we set $M_0 := M$. Now let $i = j + 1$ be given. Let $p' \in \gS (M_j)$ be the nonforking extension of $p$. By uniqueness of saturated models, there exists $f: M_j \cong_{N} M_0$. By uniqueness of nonforking extension, $f (p') = p$. By Lemma \ref{perp-fund}.(\ref{perp-fund-3}), there exists $M' \gea M_0$ in $\Ksatp{\lambda}_\lambda$ so that $p$ is omitted in $M'$. Let $M_{j + 1} := f^{-1}[M']$. Then $p'$ is omitted in $M_{j + 1}$. Since $p$ is minimal, $p$ is omitted in $|M_{j + 1}| \backslash |M_j|$, and hence by induction in $M_{j + 1}$.
  \end{itemize}
\end{proof}

An issue in transferring unidimensionality up is that we do not have a converse to Lemma \ref{perp-nf-1} (see Question \ref{perp-nf-question}), so we will ``cheat''and use the following transfer which follows from the proof of \cite[Theorem 6.3]{tamenessthree} (recall that we are assuming Hypothesis \ref{sec-5-hyp}).

\begin{fact}\label{upward-transfer-2}
  If $\K$ is categorical in $\LS (\K)$ and $\LS (\K)^+$, then $\K$ is categorical in all $\mu \in [\LS (\K), \theta]$.
\end{fact}

For the convenience of the reader, we have repeated Hypothesis \ref{sec-5-hyp} in the statement of the next two theorems.

\begin{thm}\label{unidim-transfer}
  Let $\theta > \LS (\K)$ be a cardinal or $\infty$. and let $\F := [\LS (\K), \theta)$. Let $\s$ be a type-full good $\F$-frame on $\K_{\F}$.
    
  Assume that $\K$ is categorical in $\LS (\K)$. Let $\lambda$ and $\mu$ both be in $\F$. Then $\s$ is $\lambda$-unidimensional if and only if $\s$ is $\mu$-unidimensional. 
\end{thm}
\begin{proof}
  Without loss of generality, $\mu < \lambda$. We first show that if $\s$ is \emph{not} $\mu$-unidimensional, then $\s$ is \emph{not} $\lambda$-unidimensional. Assume that $\s$ is not $\mu$-unidimensional. Let $M_0 \in \Ksatp{\mu}_\mu$ and let $p \in \gS (M_0)$ be minimal (exists by density of minimal types). By definition (and the proof of Lemma \ref{technical-multidim-equiv}), there exists $q \in \gS (M_0)$ so that $p \perp q$. Now let $M \in \Ksatp{\lambda}_\lambda$ be such that $M_0 \lea M$. Let $p', q'$ be the nonforking extensions to $M$ of $p$ and $q$ respectively. By Lemma \ref{perp-nf-1}, $p' \perp q'$. By Lemma \ref{min-unidim}, $\s$ is not $\lambda$-unidimensional.

  Conversely, assume that $\s$ is $\mu$-unidimensional. By the first part, $\s$ is $\LS (\K)$-unidimensional. By Theorem \ref{unidim-categ}, $\K$ is categorical in $\LS (\K)^+$. By Fact \ref{upward-transfer-2}, $\K$, and hence $\Ksatp{\lambda}$, is categorical in $\lambda^+$. By Theorem \ref{unidim-categ} again, $\s$ is $\lambda$-unidimensional.
\end{proof}

We obtain the promised categoricity transfer. Note that it suffices to assume that $\Ksatp{\lambda}$ (instead of $\K$) is categorical in $\lambda^+$.

\begin{thm}\label{good-categ-transfer}
  Let $\theta > \LS (\K)$ be a cardinal or $\infty$. and let $\F := [\LS (\K), \theta)$. Let $\s$ be a type-full good $\F$-frame on $\K_{\F}$.
  
  Assume that $\K$ is categorical in $\LS (\K)$ and let $\lambda \in \F$. If $\Ksatp{\lambda}$ is categorical in $\lambda^+$, then $\K$ is categorical in every $\mu \in [\LS (\K), \theta]$.
\end{thm}
\begin{proof}
  Assume that $\Ksatp{\lambda}$ is categorical in $\lambda^+$. We prove by induction on $\mu \in [\LS (\K), \theta]$ that $\K$ is categorical in $\mu$. By assumption, $\K$ is categorical in $\LS (\K)$. Now let $\mu \in (\LS (\K), \theta]$ and assume that $\K$ is categorical in every $\mu_0 \in [\LS (\K), \mu)$. If $\mu$ is limit, then it is easy to see that every model of size $\mu$ must be saturated, hence $\K$ is categorical in $\mu$. Now assume that $\mu$ is a successor, say $\mu = \mu_0^+$ for $\mu_0 \in \F$. By assumption, $\Ksatp{\lambda}$ is categorical in $\lambda^+$. By Theorem \ref{unidim-categ}, $\s$ is $\lambda$-unidimensional. By Theorem \ref{unidim-transfer}, $\s$ is $\mu_0$-unidimensional. By Theorem \ref{unidim-categ}, $\Ksatp{\mu_0}$ is categorical in $\mu_0^+$. By the induction hypothesis, $\K$ is categorical in $\mu_0$, hence $\Ksatp{\mu_0} = \K_{\ge \mu_0}$, so $\K$ is categorical in $\mu_0^+ = \mu$, as desired.
\end{proof}

\part{Applications}

\section{Background}\label{background-sec-2}

The definition of superstability below is already implicit in \cite{shvi635} and several variants were studied in, e.g.\ \cite{vandierennomax, gvv-mlq, indep-aec-apal, bv-sat-v3, gv-superstability-v3, vv-symmetry-transfer-v3}. We will use the statement from \cite[Definition 10.1]{indep-aec-apal}.

\begin{defin}\label{ss assm}
  $\K$ is \emph{$\mu$-superstable} (or \emph{superstable in $\mu$}) if:

  \begin{enumerate}
    \item $\mu \ge \LS (\K)$.
    \item $\K_\mu$ is nonempty, has amalgamation, joint embedding, and no maximal models.
    \item $\K$ is stable in $\mu$, and:
    \item\label{split assm} $\mu$-splitting in $\K$ satisfies the following
  locality property: for every limit ordinal $\delta < \mu^+$ and every increasing continuous sequence $\seq{M_i : i \le \delta}$ in $\K_\mu$ with $M_{i + 1}$ universal over $M_i$ for all $i < \delta$, if $p \in \gS (M_\delta)$, then there exists $i < \delta$ so that $p$ does not $\mu$-split over $M_i$.
  \end{enumerate}
\end{defin}

We will use the following without comments. See \cite[Fact 4.8.(2)]{vv-symmetry-transfer-v3} for a proof.

\begin{fact}\label{ss-good-frame}
  If $\s$ is a type-full good $\lambda$-frame on $\K$, then $\K$ is $\lambda$-superstable.
\end{fact}

In the setup of this paper, superstability follows from categoricity. If (as will be the case in most of this paper) the AEC is categorical in a successor, this is due to Shelah and appears as \cite[Lemma 6.3]{sh394}. The heart of the proof in the general case appears as \cite[Theorem 2.2.1]{shvi635} and the result is stated for classes with amalgamation in \cite[Corollary 6.3]{gv-superstability-v3}. A full axiomatic proof appears in \cite{shvi-notes-v1}.

\begin{fact}[The Shelah-Villaveces theorem]\label{shvi} 
  If $\K$ has amalgamation, no maximal models, and is categorical in a $\lambda > \LS (\K)$, then $\K$ is $\LS (\K)$-superstable.
\end{fact}

Together with superstability, a powerful tool is the symmetry property for splitting, first isolated by VanDieren \cite{vandieren-symmetry-v4-toappear}:
\begin{defin}\label{sym defn}
Let $\mu \ge \LS (\K)$ and assume that $\K$ has amalgamation in $\mu$. $\K$ exhibits \emph{symmetry for $\mu$-splitting} (or \emph{$\mu$-symmetry} for short) if  whenever models $M,M_0,N\in\K_\mu$ and elements $a$ and $b$  satisfy the conditions \ref{limit sym cond}-\ref{last} below, then there exists  $M^b$  a limit model over $M_0$, containing $b$, so that $\gtp(a/M^b)$ does not $\mu$-split over $N$.
\begin{enumerate} 
\item\label{limit sym cond} $M$ is universal over $M_0$ and $M_0$ is a limit model over $N$.
\item\label{a cond}  $a\in M\backslash M_0$.
\item\label{a non-split} $\gtp(a/M_0)$ is non-algebraic and does not $\mu$-split over $N$.
\item\label{last} $\gtp(b/M)$ is non-algebraic and does not $\mu$-split over $M_0$. 
\end{enumerate}
\end{defin}

When the class is tame, symmetry follows from superstability \cite[Corollary 6.9]{vv-symmetry-transfer-v3} and superstability transfers upward \cite[Proposition 10.10]{indep-aec-apal} hence they both hold everywhere:

\begin{fact}\label{tame-sym-ss}
  If $\K$ has amalgamation, is $\LS (\K)$-tame, and is $\LS (\K)$-superstable, then $\K$ is superstable and has symmetry in every $\mu \ge \LS (\K)$. 
\end{fact}

One consequence of the symmetry property is given by the following more precise statement of Fact \ref{satfact} (see \cite[Lemma 2.20]{vv-symmetry-transfer-v3} and \cite[Theorem 1]{vandieren-chainsat-apal}):

\begin{fact}\label{sym-unionsat}
  Assume that $\K$ has amalgamation. Let $\chi > \LS (\K)$. If for every $\mu \in [\LS (\K), \chi)$, $\K$ is superstable in $\mu$ and $\mu^+$ and has symmetry in $\mu^+$, then $\Ksatp{\chi}$ is an AEC with $\LS (\Ksatp{\chi}) = \chi$.
\end{fact}

We will use the following consequences of categoricity in a suitable cardinal. The notation from Chapter 14 of \cite{baldwinbook09} will come in handy:

\begin{notation}\label{hanf-def}
  For $\lambda$ an infinite cardinal, let $\hanf{\lambda} := \ehanf{\lambda}$. When $\K$ is fixed, we write $H_1$ for $\hanf{\LS (\K)}$ and $H_2$ for $\hanf{H_1} = \hanf{\hanf{\LS (\K)}}$.
\end{notation}

\begin{fact}\label{sym-from-categ}
  Assume that $\K$ has amalgamation and no maximal models. Let $\lambda > \LS (\K)$ be such that $\K$ is categorical in $\lambda$.

  \begin{enumerate}
  \item \cite[Corollary 7.2]{vv-symmetry-transfer-v3} If $\lambda \ge H_1$ or the model of size $\lambda$ is $\LS (\K)^+$-saturated (e.g.\ if $\cf{\lambda} > \LS (\K)$), then $\K$ has $\LS (\K)$-symmetry.
  \item \cite[Corollary 7.4]{vv-symmetry-transfer-v3} If $\lambda \ge \hanf{\LS (\K)^+}$, then the model of size $\lambda$ is $\LS (\K)^+$-saturated.
  \end{enumerate}
\end{fact}

As a special case, we obtain the following result that is already stated in \cite[Claim I.6.7]{sh394}.

\begin{cor}\label{cor-sat-categ}
  Assume that $\K$ has amalgamation and no maximal models. Let $\lambda > \mu > \LS (\K)$ be such that $\K$ is categorical in $\lambda$. If the model of size $\lambda$ is $\mu^+$-saturated, then $\Ksatp{\mu}$ is an AEC with $\LS (\Ksatp{\mu}) = \mu$.
\end{cor}
\begin{proof}
  By Fact \ref{shvi}, $\K$ is superstable in every $\chi \in [\LS (\K), \lambda)$. By Fact \ref{sym-from-categ}, $\K$ has symmetry in every $\chi \in [\LS (\K), \mu]$. Now apply Fact \ref{sym-unionsat}.
\end{proof}

The following fact tells us that we can often assume without loss of generality that a categorical AECs with amalgamation also has no maximal models. The proof is folklore (see e.g.\ \cite[Proposition 10.13]{indep-aec-apal}).

\begin{fact}\label{jep-decomp}
  Assume that $\K$ has amalgamation. Let $\lambda \ge \LS (\K)$ be such that $\K$ has joint embedding in $\lambda$. Then there exists $\chi < \beth_{\left(2^{\LS (\K)}\right)^+}$ and an AEC $\K^\ast$ such that:

    \begin{enumerate}
    \item $\K^\ast \subseteq \K$ and $\K^\ast$ has the same strong substructure relation as $\K$.
    \item $\LS (\K^\ast) = \LS (\K)$.
    \item $\K^\ast$ has amalgamation, joint embedding, and no maximal models.
    \item $\K_{\ge \min (\lambda, \chi)} = \K^\ast_{\ge \min (\lambda, \chi)}$.
  \end{enumerate}
\end{fact}

Let us also recall the definition of tameness (first isolated in \cite{tamenessone}) and weak tameness (already implicit in \cite{sh394}). We use the notation from \cite[Definition 11.6]{baldwinbook09}

\begin{defin}[Tameness]\label{weak-tameness-def}
  Let $\chi, \mu$ be cardinals with $\LS (\K) \le \chi \le \mu$. Assume that $\K_{[\chi, \mu]}$ has amalgamation. $\K$ is \emph{$(\chi, \mu)$-tame} if for any $M \in \K_\mu$, any $p, q \in \gS (M)$, if $p \neq q$, there exists $M_0 \in \K_{\chi}$ with $M_0 \lea M$ and $p \rest M_0 \neq q \rest M_0$. For $\theta \ge \mu$, $\K$ is \emph{$(\chi, <\theta)$-tame} if it is $(\chi, \mu)$-tame for every $\mu \in [\chi, \theta)$. $(\chi, \le \theta)$-tame means $(\chi, <\theta^+)$-weakly tame. Finally, $\K$ is \emph{$\chi$-tame} if it is $(\chi, \mu)$-weakly tame for every $\mu \ge \chi$. We similarly define variations such as $(<\chi, \mu)$-tame.

    Let us also define $\K$ is \emph{$(\chi, \mu)$-weakly tame} to mean that for any \emph{saturated} $M \in \K_\mu$, any $p, q \in \gS (M)$, if $p \neq q$, there exists $M_0 \in \K_{\chi}$ with $M_0 \lea M$ and $p \rest M_0 \neq q \rest M_0$. Define variations such as $(\chi, <\mu)$-weakly tame as above.
\end{defin}
\begin{remark}
  Tameness says that type over \emph{any} models are determined by their small restrictions. Weak tameness says that only types over \emph{saturated} models have this property. While there is no known example of an AEC that is weakly tame but not tame, it is known that weak tameness follows from categoricity in a suitable cardinal (but the corresponding result for non-weak tameness is open, see \cite[Conjecture 1.5]{tamenessthree}), see Section \ref{weak-tameness-sec}.
\end{remark}

It was noticed in \cite{ss-tame-jsl} (and further improvements in \cite[Section 10]{indep-aec-apal} or \cite[Corollary 6.14]{vv-symmetry-transfer-v3}) that tameness can be combined with superstability to build a good frame. This can also be done using only weak tameness:

\begin{fact}[Theorem 6.4 in \cite{vv-symmetry-transfer-v3}]\label{good-frame-weak-tameness}
  Let $\lambda > \LS (\K)$. Assume that $\K$ is superstable in every $\mu \in [\LS (\K), \lambda]$ and has $\lambda$-symmetry.
  
  If $\K$ is $(\LS (\K), \lambda)$-weakly tame, then there exists a type-full good $\lambda$-frame with underlying class $\Ksatp{\lambda}_\lambda$ (so in particular, $\Ksatp{\lambda}_\lambda$ is the initial segment of an AEC).
\end{fact}

Once we have a good $\lambda$-frame, we can enlarge it so that the forking relation works over larger models.

\begin{fact}[Corollary 6.9 in \cite{tame-frames-revisited-v5}]\label{frame-upward-transfer}
  Let $\theta > \lambda \ge \LS (\K)$. Let $\F := [\lambda, \theta)$. Assume that $\K_{\F}$ has amalgamation. Let $\s$ be a type-full good $\lambda$-frame on $\K_{\lambda}$. If $\K$ is $(\lambda, <\theta)$-tame, then there exists a type-full good $\F$-frame $\s'$ extending $\s$: $\s' \rest \K_\lambda = \s$.
\end{fact}

Assuming only weak tameness, we can show that if $\s$ is a (type-full) good $\mu$-frame and $\s'$ is a good $\lambda$-frame with $\mu < \lambda$ and the underlying class of $\s'$ is the saturated models in the underlying class of $\s$, then forking in $\s'$ can be described in terms of forking in $\s$. This is proven as Theorem \ref{shrinking-compat} in the appendix and is used to replace tameness by weak tameness in the main theorem.

%% We obtain\svnote{Needed?}:

%% \begin{prop}\label{frame-existence}
%%   Let $\K$ be an AEC with amalgamation and no maximal models. Assume that $\K$ is categorical in a successor $\lambda > \LS (\K)^+$. 

%%   \begin{enumerate}
%%     \item If $\mu \in (\LS (\K), \lambda)$ is such that $\K$ is $(\LS (\K), \mu)$-weakly tame, then there exists a type-full good $\mu$-frame with underlying class $\Ksatp{\mu}_\mu$.
%%     \item Let $\theta > \LS (\K)^+$. If $\K$ is $(\LS (\K), <\theta)$-tame, then there exists a type-full good $[\LS (\K)^+, \theta)$-frame with underlying class $\Ksatp{\LS (\K)^+}_{[\LS (\K)^+, \theta)}$.
%%   \end{enumerate}
%% \end{prop}
%% \begin{proof} 
%%   By Fact \ref{sym-from-categ}, $\K$ is $\mu$-superstable and has $\mu$-symmetry for every $\mu \in [\LS (\K), \lambda)$. Now:
%%   \begin{enumerate}
%%     \item By Fact \ref{good-frame-weak-tameness}.
%%     \item Use the previous part with $\mu = \LS (\K)^+$, then apply Fact \ref{frame-upward-transfer}. Note that by (the proof of) \cite[Theorem 6.8]{vv-symmetry-transfer-v2}, $\Ksatp{\LS (\K)^+}_{[\LS (\K)^+, \theta)}$ is the initial segment of an AEC.
%%   \end{enumerate}
%% \end{proof}

\section{Weak tameness from categoricity}\label{weak-tameness-sec}

We quote a result of Shelah from \cite[Chapter IV]{shelahaecbook} on deriving weak tameness from categoricity and combine it with the corresponding results in \cite{vv-symmetry-transfer-v3}. We derive a small improvement on some of the Hanf numbers, positively answering a question of Baldwin \cite[Question 11.16]{baldwinbook09} (see also \cite[Remark 14.15]{baldwinbook09}) which asked whether it was possible to obtain $\chi$-weak tameness for some $\chi < H_1$ rather than $(<H_1)$-weak tameness. We give two applications in AECs with amalgamation and no maximal models that are categorical in a high-enough cardinal that is still \emph{below} the Hanf number: the construction of a good frame and a non-trivial restriction on the categoricity spectrum.

The following appears as \cite[Main Claim II.2.3]{sh394} (a simplified and improved argument is in \cite[Theorem 11.15]{baldwinbook09}):

\begin{fact}\label{weak-tameness-from-categ-fact-0}
  Assume that $\K$ has amalgamation. Let $\lambda > \mu \ge H_1$. Assume that $\K$ is categorical in $\lambda$, and the model of cardinality $\lambda$ is $\mu^+$-saturated. Then there exists $\chi < H_1$ such that $\K$ is $(\chi, \mu)$-weakly tame.
\end{fact}

As opposed to Fact \ref{weak-tameness-from-categ-fact-0}, the following also applies when the categoricity cardinal is below the Hanf number.

\begin{fact}[Claim IV.7.2 in \cite{shelahaecbook}]\label{weak-tameness-categ-2}
  Let $\mu > \LS (\K)$. If:

  \begin{enumerate}
    \item $\K_{<\mu}$ has amalgamation.
    \item $\cf{\mu} > \LS (\K)$.
    \item\label{weak-tameness-cond-3} $\Phi$ is a proper for linear orders, and if $\theta \in (\LS (\K), \mu)$, $I$ is a $\theta$-wide\footnote{A linear order is \emph{$\theta$-wide} if for every $\theta_0 < \theta$, $I$ contains an increasing sequence of length $\theta_0^+$, see \cite[Definition IV.0.14.(1)]{shelahaecbook}.} linear order, then $\text{EM}_{\tau (\K)} (I, \Phi)$ is $\theta$-saturated.
  \end{enumerate}

  Then there exists $\chi \in (\LS (\K), \mu)$ such that $\K$ is $(\chi, <\mu)$-weakly tame.
\end{fact}

Condition (\ref{weak-tameness-cond-3}) in Fact \ref{weak-tameness-categ-2} can be derived from categoricity if the model in the categoricity cardinal is sufficiently saturated. This is implicit in \cite{sh394} and appears as \cite[Lemma 10.11]{baldwinbook09}.

\begin{fact}\label{saturation-fact}
  If $\K$ has amalgamation and no maximal models, $\mu > \LS (\K)$, $\K$ is categorical in a $\lambda \ge \mu$, so that the model of size $\lambda$ is $\mu$-saturated, then for every $\Phi$ proper for linear orders, if $\theta \in (\LS (\K), \mu)$ and $I$ is a $\theta$-wide linear order, we have that $\text{EM}_{\tau (\K)} (I, \Phi)$ is $\theta$-saturated.
\end{fact}

Combining these facts, we obtain the following result. Note that the second part is a slight improvement on Fact \ref{weak-tameness-from-categ-fact-0}, as the model of size $\lambda$ is allowed to be $H_1$-saturated. Moreover the amount of weak tameness $\chi$ can be chosen independently of $\mu$:

\begin{thm}\label{technical-lem}
  Assume that $\K$ has amalgamation. Let $\lambda  > \LS (\K)$ be such that $\K$ is categorical in $\lambda$.

  \begin{enumerate}
    \item Let $\mu$ be a limit cardinal such that $\cf{\mu} > \LS (\K)$. If $\K$ has no maximal models and the model of size $\lambda$ is $\mu$-saturated, then there exists $\chi < \mu$ such that $\K$ is $(\chi, <\mu)$-weakly tame.
    \item If the model of size $\lambda$ is $H_1$-saturated, then there exists $\chi < H_1$ such that whenever $\mu \ge H_1$ is so that the model of size $\lambda$ is $\mu$-saturated, we have that $\K$ is $(\chi, <\mu)$-weakly tame.
  \end{enumerate}
\end{thm}
\begin{proof} \
  \begin{enumerate}
    \item By Fact \ref{weak-tameness-categ-2} (using Fact \ref{saturation-fact} to see that (\ref{weak-tameness-cond-3}) is satisfied).
    \item Without loss of generality (Fact \ref{jep-decomp}), $\K$ has no maximal models. By the first part (with $\mu$ there standing for $H_1$ here), there exists $\chi < H_1$ such that $\K$ is $(\chi, <H_1)$-weakly tame. Now assume that the model of size $\lambda$ is $\mu$-saturated, for $\mu > H_1$. Let $\mu' \in [H_1, \mu)$. We show that $\K$ is $(\chi, \mu')$-weakly tame. By Fact \ref{weak-tameness-from-categ-fact-0} (with $\mu$ there standing for $\mu'$ here), there exists $\chi' < H_1$ such that $\K$ is $(\chi', \mu')$-weakly tame. In particular, $\K$ is $(H_1, \mu')$-weakly tame. Now by Corollary \ref{cor-sat-categ}, $\Ksatp{H_1}$ is an AEC with $\LS (\Ksatp{H_1}) = H_1$. Thus we can combine $(\chi, H_1)$-weak and $(H_1, \mu')$-weak tameness to get $(\chi, \mu)$-weak tameness, as desired.
  \end{enumerate}
\end{proof}

We give two applications of (the first part of) Theorem \ref{technical-lem}. First, we obtain an improvement on the Hanf number for the construction of a good frame in \cite[Corollary 7.9]{vv-symmetry-transfer-v3} ($\mu$ below can be less than $H_1$, e.g.\ $\mu = \aleph_{\LS (\K)^+}$).

\begin{thm}\label{good-frame-improvement}
  Assume that $\K$ has amalgamation and no maximal models. Let $\mu$ be a limit cardinal such that $\cf{\mu} > \LS (\K)$ and assume that $\K$ is categorical in a $\lambda \ge \mu$. If the model of size $\lambda$ is $\mu$-saturated, then there exists $\chi < \mu$ such that for all $\mu_0 \in [\chi, \mu)$, there is a good $\mu_0$-frame on $\Ksatp{\mu_0}$.
\end{thm}
\begin{proof}
By Fact \ref{sym-from-categ}, $\K$ has $\mu_0$-symmetry for every $\mu_0 < \chi$. By Fact \ref{shvi}, $\K$ is also superstable in every $\mu_0 \in [\LS (K), \lambda)$. Now by Theorem \ref{technical-lem}, there exists $\chi < \mu$ so that $\K$ is $(\chi, <\mu)$-weakly tame. We finish by applying Fact \ref{good-frame-weak-tameness}.
\end{proof}

Second, we can study the categoricity spectrum \emph{below} the Hanf number of an AEC with amalgamation and no maximal models. While it is known that the categoricity spectrum in such AECs must be a closed set (see the proof of \cite[Corollary 4.3]{tamenesstwo}), we show (in ZFC) that there are other restrictions:

\begin{thm}\label{categ-below-hanf}
  Assume that $\K$ has amalgamation and arbitrarily large models. Let $\mu$ be a limit cardinal such that $\cf{\mu} > \LS (\K)$. If $\K$ is categorical in unboundedly many successor cardinals below $\mu$, then there exists $\mu_0 < \mu$ such that $\K$ is categorical in every $\lambda \in [\mu_0, \mu]$.

  In particular (setting $\mu := \aleph_{\LS (\K)^+}$), if $\K$ is categorical in $\aleph_{\alpha + 1}$ for unboundedly many $\alpha < \LS (\K)^+$, then there exists $\alpha_0 < \LS (\K)^+$ such that $\K$ is categorical in $\aleph_\beta$ for every $\beta \in [\alpha_0, \LS (\K)^+]$.
\end{thm}

Before starting the proof, we make a remark:

\begin{remark}\label{gv-upward-rmk}
  Fact \ref{gv-upward} generalizes to AECs that are only $\LS (\K)$-weakly tame, or just $(\LS (\K), <\theta)$-weakly tame (in the second case, we can only conclude categoricity up to and including $\theta$). This is implicit in \cite{tamenesstwo, tamenessthree} and stated explicitly in Chapter 13 of \cite{baldwinbook09}.
\end{remark}

\begin{proof}[Proof of Theorem \ref{categ-below-hanf}]
  Without loss of generality (Fact \ref{jep-decomp}), $\K$ has no maximal models. By amalgamation, every model of size $\mu$ is saturated. In particular $\K$ is categorical in $\mu$. By Theorem \ref{technical-lem} (with $\lambda, \mu$ there standing for $\mu, \mu$ here), there exists $\mu_0' < \mu$ such that $\K$ is $(\mu_0', <\mu)$-weakly tame. By making $\mu_0'$ bigger if necessary, we can assume without loss of generality that $\mu_0' > \LS (\K)$ and $\K$ is categorical in $\mu_0 := (\chi_0')^+$. By the upward categoricity transfer of Grossberg and VanDieren (Fact \ref{gv-upward}, keeping in mind Remark \ref{gv-upward-rmk}), $\K$ is categorical in every $\lambda \in [\mu_0, \mu]$.
\end{proof}

\section{Shelah's omitting type theorem}\label{shelah-sec}

In this section, we give a nonlocal proof of the improvement of the bounds in \cite{sh394} from $\beth_{H_1}$ to $H_1$ using the methods of \cite{sh394}. We will present a more powerful local proof in the next sections. We also give several partial categoricity transfers in AECs with amalgamation, including Theorem \ref{omitting-categ-transfer} which says that in a tame AEC with amalgamation, categoricity in \emph{some} cardinal (above the tameness cardinal) implies categoricity in a proper class of cardinals. The main tool is a powerful generalization of Morley's omitting type theorem (Fact \ref{shelah-omit}), an early form of which appears in \cite{makkaishelah}.

All throughout, we assume:

\begin{hypothesis}
  $\K$ is an AEC with amalgamation.
\end{hypothesis}

As a motivation, we first state Morley's omitting type theorem for AECs \cite[II.1.10]{sh394}. We state a slightly stronger conclusion (replacing $H_1$ by some $\chi < H_1$) that is implicit e.g.\ in \cite{sh394} but to the best of our knowledge, a proof of this stronger result has not appeared in print before. We include a proof (similar to the proof of \cite[Theorem 5.4]{bg-v10}, though there is an additional step involved) for the convenience of the reader.

\begin{fact}[Morley's omitting type theorem for AECs]\label{omitting-type}
  Let $\lambda > \LS (\K)$. If every model in $\K_\lambda$ is $\LS (\K)^+$-saturated, then there exists $\chi < H_1$ such that every model in $\K_{\ge \chi}$ is $\LS (\K)^+$-saturated.
\end{fact}
\begin{proof}[Proof sketch]
  Without loss of generality (Fact \ref{jep-decomp}), $\K$ has no maximal models. Suppose the conclusion fails. Then for every $\chi \in [\LS (\K), H_1)$, there exists $M_{\chi} \in \K_{\chi}$ which is not $\LS (\K)^+$-saturated. Pick witnesses $M_{0, \chi} \lea M_\chi$ and $p_{\chi} \in \gS (M_{0, \chi})$ such that $\|M_{0, \chi}\| = \LS (\K)$ and $M_{\chi}$ omits $p_{\chi}$. Now there are only $2^{\LS (K)}$ isomorphism types of Galois types over models of size $\LS (\K)$, and $\cf{H_1} = \left(2^{\LS (\K)}\right)^+ > 2^{\LS (\K)}$, so there exists $N \in \K_{\LS (\K)}$, $p \in \gS (M)$, and an unbounded $S \subseteq [\LS (\K), H_1)$ such that for all $\chi \in S$, $p_{\chi}$ is isomorphic to $p$ (in the natural sense). Look at the AEC $\K_{\neg p}$ of all the models of $\K$ omitting $p$, with constants added for $N$ (see e.g.\ the definition of $\K^+$ in the proof of \cite[Theorem 5.4]{bg-v10}). For each $\chi \in S$, an appropriate expansion of a copy of $M_{\chi}$ is in $\K_{\neg p}$. $\K_{\neg p}$ has Löwenheim-Skolem-Taski number $\LS (\K)$, so by Shelah's presentation theorem and Morley's omitting type theorem (for first-order theories), $\K_{\neg p}$ has arbitrarily large models, contradicting the assumptions on $\lambda$.
\end{proof}

A generalization of Fact \ref{omitting-type} is what we call Shelah's omitting type theorem. The statement appears (in a more general form) in \cite[Lemma II.1.6]{sh394}, but the full proof (for models of an $\Ll_{\kappa, \omega}$ theory, $\kappa$ a strongly compact cardinal) can already be found in \cite[Proposition 3.3]{makkaishelah} (see also Will Boney's note \cite{boney-shelah-omitting}). We state a simplified version:

\begin{fact}[Shelah's omitting type theorem]\label{shelah-omit}
  Let $M_0 \lea M$ both be in $\K_{\ge \LS (\K)}$ and let $p \in \gS (M_0)$. Assume that $M$ omits $p / E_{\LS (\K)}$. That is, for every $a \in |M|$, there is $M_0' \lea M_0$ with $\|M_0'\| = \LS (\K)$ such that $\gtp (a / M_0'; M) \neq p \rest M_0'$. 

  If $ \beth_{(2^{\LS (\K)})^+} (\|M_0\|) \le \|M\|$, then there is a non-$\LS (\K)^+$-saturated model in every cardinal.
\end{fact}

Note that taking $\|M_0\| = \LS (\K)$, we recover Morley's omitting type theorem for AECs. Note also that when $\K$ is $\LS (\K)$-tame, $M$ above omits $p / E_{\LS (\K)}$ if and only if $M$ omits $p$. The following two direct consequences are implicit in \cite{sh394}.

\begin{lem}\label{omitting-type-technical}
  Let $\LS (\K) < \lambda$ and let $\LS (\K) \le \chi < \mu$. Assume that:
  
  \begin{enumerate}
  \item $\beth_{\left(2^{\LS (\K)}\right)^+} (\chi) \le \mu$.
  \item $\K$ is $(\LS (\K), \le \chi)$-weakly tame.
  \item For every $\chi_0 \in [\LS (\K)^+, \chi]$, $\Ksatp{\chi_0}$ is an AEC with $\LS (\Ksatp{\chi_0}) = \chi_0$.
  \end{enumerate}

  If every model in $\K_{\lambda}$ is $\LS (\K)^+$-saturated, then every model in $\K_{\mu}$ is $\chi^+$-saturated.
\end{lem}
\begin{proof}
  Assume that every model in $\K_{\lambda}$ is $\LS (\K)^+$-saturated. We prove that for all $\chi_0 \in [\LS (\K), \chi]$, every model in $\K_{\mu}$ is $\chi_0^+$-saturated. We proceed by induction on $\chi_0$. If $\chi_0 = \LS (\K)$,  take $M \in \K_{\mu}$ and $M_0 \lea M$ of size $\chi_0$. If there is a type over $M_0$ omitted in $\K_{\mu}$, then by Fact \ref{shelah-omit}, there is a non-$\LS (\K)^+$-saturated model of size $\lambda$, a contradiction. Therefore every model in $\K_{\mu}$ is $\LS (\K)^+$-saturated.

  Now let $\chi_0 > \LS (\K)$ and assume inductively that every model in $\K_{\mu}$ is $\chi_0$-saturated. Let $M \in \K_{\mu}$ and let $M_0 \lea M$ have size $\chi_0$. Since $\LS (\Ksatp{\chi_0}) = \chi_0$ and $M$ is $\chi_0$-saturated, we take enlarge $M_0$ if necessary to assume without loss of generality that $M_0$ is $\chi_0$-saturated. Let $p \in \gS (M_0)$. Then by the weak tameness hypothesis $p / E_{\LS (\K)} = p$. So if $M$ omits $p$, then by Fact \ref{shelah-omit}, $\K_\lambda$ has a non-$\LS (\K)^+$-saturated model, a contradiction, so $M$ realizes $p$, as needed.
\end{proof}
\begin{remark}
  We only use above that $\LS (\Ksatp{\chi_0}) = \chi_0$, not that $\Ksatp{\chi_0}$ is an AEC. That is, we only use that for $M \in \Ksatp{\chi_0}$ and $A \subseteq |M|$, there exists a $\chi_0$-saturated $M_0 \lea M$ such that $\|M_0\| \le |A| + \chi_0$.
\end{remark}
\begin{lem}\label{omitting-type-technical-2}
  Let $\LS (\K) < \mu < \lambda$. Assume that:
  
  \begin{enumerate}
  \item $\K$ is categorical in $\lambda$.
  \item $\mu = \beth_{\delta}$, for some limit ordinal $\delta$ divisible by $\left(2^{\LS (\K)}\right)^+$.
  \item $\K$ is $(\LS (\K), <\mu)$-weakly tame.
  \end{enumerate}

  If the model of size $\lambda$ is $\mu$-saturated, then $\K$ is categorical in $\mu$.
\end{lem}
\begin{proof}
  Without loss of generality (by Fact \ref{jep-decomp}), $\K$ has joint embedding and no maximal models. Thus it is enough to show that every model of size $\mu$ is saturated. Let $\chi \in [\LS (\K), \mu)$. We have to check that the hypotheses of Lemma \ref{omitting-type-technical} hold. The only problematic part is to see that $\Ksatp{\chi}$ is an AEC with $\LS (\Ksatp{\chi}) = \chi$ (when $\chi > \LS (\K)$). But this holds by Corollary \ref{cor-sat-categ}.
\end{proof}

We obtain the following downward transfer result that slightly improves on \cite[Corollary 7.7]{vv-symmetry-transfer-v3}:

\begin{cor}\label{downward-transfer}
  Let $\LS (\K) < \mu < \lambda$ be such that:

  \begin{enumerate}
  \item $\K$ is categorical in $\lambda$.
  \item $\mu = \beth_{\delta}$, for some limit ordinal $\delta$ divisible by $H_1$.
  \end{enumerate}

  If the model of size $\lambda$ is $\mu$-saturated (e.g.\ if $\cf{\lambda} \ge \mu$ or by Fact \ref{sym-from-categ} if $\lambda \ge \sup_{\mu_0 < \mu} \hanf{\mu_0^+}$), then $\K$ is categorical in $\mu$.
\end{cor}
\begin{proof}
  By Theorem \ref{technical-lem}, there exists $\chi < H_1$ such that $\K$ is $(\chi, <\mu)$-weakly tame. Now apply Lemma \ref{omitting-type-technical-2} to $\K_{\ge \chi}$.
\end{proof}

When tameness holds instead of weak tameness, we obtain the following generalization of the second main result of \cite{makkaishelah}:

\begin{thm}\label{omitting-categ-transfer}
  If $\K$ is $\LS (\K)$-tame, has arbitrarily large models, and is categorical in a $\lambda > \LS (\K)$, then $\K$ is categorical in all cardinals of the form $\beth_{\delta}$, where $\left(2^{\LS (\K)}\right)^+$ divides $\delta$.
\end{thm}
\begin{proof}
  Without loss of generality (Fact \ref{jep-decomp}), $\K$ has no maximal models. Let $\delta$ be a limit ordinal divisible by $\left(2^{\LS (\K)}\right)^+$. Let $\mu := \beth_{\delta}$. We prove that every model in $\K_{\mu}$ is saturated. Observe first that for every $\chi \ge \LS (\K)^+$, $\Ksatp{\chi}$ is an AEC with $\LS (\Ksatp{\chi}) = \chi$. This follows from Facts \ref{shvi}, \ref{tame-sym-ss}, and \ref{sym-unionsat}. In particular, the model of size $\lambda$ is $\LS (\K)^+$-saturated. Therefore for each $\chi < \mu$, Lemma \ref{omitting-type-technical} tells us that every model in $\K_{\mu}$ is $\chi^+$-saturated. Thus every model in $\K_{\mu}$ is saturated.
\end{proof}
\begin{remark}
  We could rely on fewer background facts by directly using Fact \ref{shelah-omit}. In that case, all that one needs to show is that the model of size $\lambda$ is $\LS (\K)^+$-saturated. This holds by combining Fact \ref{shvi} (telling us that $\K$ is $\LS (\K)$-superstable) and \cite[Theorem 5.6]{ss-tame-jsl} (saying that $\K$ is stable in $\lambda$).
\end{remark}

We can derive the desired improvements on the bounds of \cite{sh394}.

\begin{cor}\label{sh394-bounds-improvement}
  Let $\K$ be an $\LS (\K)$-tame AEC with amalgamation. If $\K$ is categorical in \emph{some successor} $\lambda \ge H_1$, then $\K$ is categorical in \emph{all} $\lambda' \ge H_1$.
\end{cor}
\begin{proof}
  By Theorem \ref{omitting-categ-transfer}, $\K$ is categorical in $H_1$. Now proceed as in the proof of the main result of \cite{sh394}, see e.g.\ \cite[Theorem 14.14]{baldwinbook09}.
\end{proof}

Note that this method does not allow us to go lower than the Hanf number, even if we know for example that $\K$ is categorical below it (Shelah's argument for transferring Vaughtian pairs is not local enough). See Corollary \ref{main-thm} for another proof.

\section{Categoricity at a successor or with primes}\label{main-thm-sec}

We apply Theorem \ref{good-categ-transfer} to categorical tame AECs with amalgamation when the categoricity cardinal is a successor or the AEC has primes (recall Definition \ref{prime-def}). All throughout, we assume:

\begin{hypothesis}\label{ap-hyp}
  $\K$ is an AEC with amalgamation.
\end{hypothesis}

\begin{lem}[Main lemma]\label{main-lem}
  Assume that $\K$ has no maximal models and is $\LS (\K)$-tame. Let $\lambda > \LS (\K)^+$ be a successor cardinal. If $\K$ is categorical in \emph{some successor} $\lambda > \LS (\K)^+$, then $\Ksatp{\LS (\K)^+}$ is categorical in all $\lambda' \ge \LS (\K)^+$.
\end{lem}
\begin{proof}
  By Fact \ref{shvi}, $\K$ is $\LS (\K)$-superstable. By Fact \ref{tame-sym-ss}, $\K$ is superstable and has symmetry in every $\mu \ge \LS (\K)$. By Facts \ref{good-frame-weak-tameness} and \ref{frame-upward-transfer}, there exists a type-full good $(\ge \LS (\K)^+)$-frame $\s$ with underlying class $\K_{\s} := \Ksatp{\LS (\K)^+}$. Moreover, $\K_{\s}$ is categorical in $\LS (\K)^+$. Thus we can apply Theorem \ref{good-categ-transfer} where $\K, \LS (\K), \theta$ there stand for $\Ksatp{\LS (\K)^+}, \LS (\K)^+, \infty$ here.
\end{proof}
\begin{cor}\label{main-cor}
  Assume that $\K$ has arbitrarily large models and is $\LS (\K)$-tame. Let $\LS (\K) < \lambda_0 < \lambda$. If $\lambda$ is a successor and $\K$ is categorical in $\lambda_0$ and $\lambda$, then $\K$ is categorical in all $\lambda' \ge \lambda_0$.
\end{cor}
\begin{proof}
  By Fact \ref{jep-decomp}, we can assume without loss of generality that $\K$ has no maximal models. By Lemma \ref{main-lem}, $\Ksatp{\LS (\K)^+}$ is categorical in all $\lambda' \ge \LS (\K)^+$. Moreover by the proof of Lemma \ref{main-lem}, $\K$ is stable in every $\mu \in [\LS (\K)^+, \lambda)$, hence the model of size $\lambda_0$ is saturated. Therefore $\Ksatp{\LS (\K)^+}_{\ge \lambda_0} = \K_{\ge \lambda_0}$, and the result follows.
\end{proof}
\begin{remark}\label{main-cor-rmk}
  We can allow $\lambda_0 = \LS (\K)$ but then the proof is more complicated: we do not know how to build a good $\LS (\K)$-frame so have to work with $\LS (\K)$-splitting. 
\end{remark}
\begin{remark}
  The case $\lambda' \ge \lambda$ in Lemma \ref{main-lem} and Corollary \ref{main-cor} is Fact \ref{gv-upward}. The contribution of this paper is the case $\lambda' < \lambda$.
\end{remark}

We deduce another proof of Corollary \ref{sh394-bounds-improvement}. We prove a slightly stronger result:

\begin{cor}\label{main-thm}
  Assume that $\K$ is $\LS (\K)$-tame and has arbitrarily large models. If $\K$ is categorical in \emph{some successor} $\lambda > \LS (\K)^+$, then there exists $\chi < H_1$ such that $\K$ is categorical in all $\lambda' \ge \min (\chi, \lambda)$.
\end{cor}
\begin{proof}
  By Lemma \ref{main-lem}, $\Ksatp{\LS (\K)^+}$ is categorical in all $\mu \ge \LS (\K)^+$. Since $\lambda$ is regular, the model of size $\lambda$ is saturated hence $\LS (\K)^+$-saturated, so every model in $\K_{\ge \lambda}$ is $\LS (\K)^+$-saturated. By Fact \ref{omitting-type}, there exists $\chi < H_1$ such that every model in $\K_{\ge \chi}$ is $\LS (\K)^+$-saturated. Thus every model in $\K_{\ge \min (\chi, \lambda)}$ is $\LS (\K)^+$-saturated, that is:

  $$
  \Ksatp{\LS (\K)^+}_{\ge \min (\chi, \lambda)} = \K_{\ge \min (\chi, \lambda)}
  $$

  The result follows.
\end{proof}
\begin{remark}
  An alternate proof (which just deduces categoricity in all $\lambda' \ge \min (\lambda, H_1)$) goes as follows. By Theorem \ref{omitting-categ-transfer}, $\K$ is categorical in $H_1$. By Fact \ref{gv-upward}, $\K$ is categorical in all $\lambda' \ge \lambda$. By Corollary \ref{main-cor} (where $\lambda_0$, $\lambda_1$ there stand for $\min (\lambda, H_1), \max (\lambda^+, H_1^+)$ here), we get that $\K$ is categorical in every $\lambda' \ge \min (\lambda, H_1)$.
\end{remark}

We can similarly deduce several consequences on tame AECs with primes. One of the main results of \cite{categ-primes-v3} was (the point compared to Shelah's downward categoricity transfer \cite{sh394} is that $\lambda$ need \emph{not} be a successor): 

\begin{fact}[Theorem 3.8 in \cite{categ-primes-v3}]\label{prime-fact}
  Assume that $\K$ has no maximal models, is $H_2$-tame, and $\K_{\ge H_2}$ has primes. If $\K$ is categorical in some $\lambda > H_2$, then it is categorical in all $\lambda' \ge H_2$.
\end{fact}

We show that we can obtain categoricity in more cardinals provided that $\K$ has more tameness:

\begin{cor}\label{improved-prime-categ}
  Assume that $\K$ is $\LS (\K)$-tame and has arbitrarily large models. Assume also that $\K$ has primes (or just that $\K_{\ge \mu}$ has prime, for some $\mu$). If $\K$ is categorical in \emph{some} $\lambda > \LS (\K)$, then $\K$ is categorical in \emph{all} $\lambda' \ge \min (\lambda, H_1)$.
\end{cor}
\begin{proof}
  Without loss of generality (Fact \ref{jep-decomp}), $\K$ has no maximal models. By Theorem \ref{omitting-categ-transfer}, $\K$ is categorical in a proper class of cardinals. By Fact \ref{prime-fact} (applied to $\K_{\ge \mu}$, where $\mu$ is such that $\K_{\ge \mu}$ has primes), $\K$ is categorical in a successor cardinal. By Corollary \ref{main-thm}, $\K$ is categorical in all $\lambda' \ge H_1$. By Corollary \ref{main-cor} (with $\lambda_0, \lambda$ there standing for $\lambda, H_1^+$ here), $\K$ is categorical also in all $\lambda' \in [\lambda, H_1^+]$.
\end{proof}
\begin{remark}
  Similarly to Corollary \ref{main-thm}, we get that there is a $\chi < H_1$ such that $\K$ is categorical in all $\lambda' \ge \chi$.
\end{remark}

Specializing to universal classes, we can improve some of the Hanf number bounds in \cite{ap-universal-v9}, obtaining in particular the full categoricity conjecture (i.e.\ the Hanf number is $H_1$) \emph{assuming amalgamation}.

\begin{cor}\label{improved-univ-categ}
  Let $\K$ be a universal class with amalgamation and arbitrarily large models. If $\K$ is categorical in \emph{some} $\lambda > \LS (\K)$, then $\K$ is categorical in \emph{all} $\lambda' \ge \min (\lambda, H_1)$.
\end{cor}
\begin{proof}
  By a result of Boney \cite{tameness-groups} (a full proof appears as \cite[Theorem 3.7]{ap-universal-v9}), $\K$ is $\LS (\K)$-tame. Also \cite[Remark 5.3]{ap-universal-v9}, $\K$ has primes. Thus we can apply Corollary \ref{improved-prime-categ}.
\end{proof}

\section{Categoricity in a limit without primes}\label{categ-limit-sec}

In this section, we give an exposition of Shelah's proof of the eventual categoricity conjecture in AECs with amalgamation \cite[Theorem IV.7.12]{shelahaecbook} (assuming the weak generalized continuum hypothesis), see Corollary \ref{shelah-ap}. We improve the threshold cardinal to $H_1$ assuming tameness (Corollary \ref{abstract-thm-3-proof}), and moreover give alternate proofs for several of the hard steps in Shelah's argument.

Most of the results of this section will use the weak generalized continuum hypothesis. We adopt the following notation:

\begin{notation}
  For a cardinal $\lambda$, $\WGCH (\lambda)$ is the statement ``$2^\lambda < 2^{\lambda^+}$''. More generally, for $S$ a class of cardinals, $\WGCH (S)$ is the statement ``$\WGCH (\lambda)$ for all $\lambda \in S$''. $\WGCH$ will stand for $\WGCH (\text{Card})$, where $\text{Card}$ is the class of all cardinals.
\end{notation}

We assume familiarity with the definitions of a \emph{weakly successful}, \emph{successful}, and \emph{$\omega$-successful} good $\lambda$-frame (see \cite[Definition III.1.1]{shelahaecbook}, and on weakly successful see Definition \ref{weakly-succ-def}). We use the notation from \cite{jrsh875}. We will say a good $\lambda$-frame is \emph{$\succp$} if it is successful and $\leap{\K_{\lambda^+}}^{\text{NF}}$ is just $\leap{\K_{\lambda^+}}$ on the saturated models in $\K_{\lambda^+}$, see \cite[Definition 6.1.4]{jrsh875}. We say a good $\lambda$-frame is \emph{$\omega$-$\succp$} if it is $\omega$-successful and $\succp$.

We first state the unpublished Claim of Shelah mentioned in the introduction. This stems from \cite[Discussion III.12.40]{shelahaecbook}. A proof should appear in \cite{sh842}. 

\begin{claim}\label{claim-xxx}
  Let $\s$ be an $\omega$-$\succp$ good $\lambda$-frame on $\K$. Assume $\WGCH([\lambda, \lambda^{+\omega}))$. If $\Ksatp{\lambda^{+\omega}}$ is categorical in some $\mu > \lambda^{+\omega}$, then $\Ksatp{\lambda^{+\omega}}$ is categorical in all $\mu' > \lambda^{+\omega}$. Moreover, for any $\mu' > \lambda$, $\Ksatp{\mu'}$ has amalgamation in $\mu'$.
\end{claim}

Next, we discuss how to obtain an $\omega$-$\succp$ good frame from a good frame. The proof of the following fact is contained in the proof of \cite[Theorem IV.7.12]{shelahaecbook} (see $\odot_4$ there). We give a full proof in Appendix \ref{weakly-succ-appendix}. Note that as opposed to the results in \cite[Section II.5]{shelahaecbook}, we do \emph{not} assume that $\K$ has few models in $\lambda^{+2}$.

\begin{fact}\label{weakly-succ-fact}
  Assume $\WGCH (\lambda)$. If $\s$ is a good $\lambda$-frame on $\K$, $\K$ is categorical in $\lambda$, has amalgamation in $\lambda^+$ and is stable in $\lambda^+$, then $\s$ is weakly successful.
\end{fact}

We can obtain the stability hypothesis and successfulness using weak tameness:

\begin{fact}[Theorem 4.5 in \cite{b-k-vd-spectrum}]\label{upward-stab-succ}
  Let $\lambda \ge \LS (\K)$, be such that $\K$ has amalgamation in $\lambda$ and $\lambda^+$, and $\K$ is stable in $\lambda$. If $\K$ is $(\lambda, \lambda^+)$-weakly tame, then $\K$ is stable in $\lambda^+$.
\end{fact}
\begin{fact}[Corollary 7.19 in \cite{jarden-tameness-apal}]\label{jarden-successful}
  If $\s$ is a weakly successful good $\lambda$-frame on $\K$, $\K$ is categorical\footnote{Note that by \cite[Claim III.1.21]{shelahaecbook}, this implies the conjugation property, so Hypothesis 6.5 in \cite{jarden-tameness-apal} is satisfied.} in $\lambda$, $\K$ has amalgamation in $\lambda^+$, and $\K$ is $(\lambda, \lambda^+)$-weakly tame, then $\s$ is $\succp$.
\end{fact}
\begin{remark}
  Although we will not need it, the converse (i.e.\ obtaining weak tameness from being $\succp$) is also true, see \cite[Theorem 7.1.13.(b)]{jrsh875}.
\end{remark}

\begin{cor}\label{getting-succp}
  Assume $\WGCH (\lambda)$. If $\s$ is a good $\lambda$-frame on $\K$, $\K$ is categorical in $\lambda$, has amalgamation in $\lambda^+$, and is $(\lambda, \lambda^+)$-weakly tame, then $\s$ is $\succp$.
\end{cor}
\begin{proof}
  By Fact \ref{upward-stab-succ}, $\K$ is stable in $\lambda^+$. By Fact \ref{weakly-succ-fact}, $\s$ is weakly successful. By Fact \ref{jarden-successful}, $\s$ is $\succp$.
\end{proof}

Using Fact \ref{good-frame-weak-tameness} to build the good frame, we obtain:

\begin{lem}\label{key-lem-categ-limit}
  Assume that $\K$ has amalgamation and no maximal models. Assume $\WGCH ([\LS (\K), \LS (\K)^{+\omega}))$ and Claim \ref{claim-xxx}. If $\K$ is categorical in a $\lambda > \LS (\K)^{+\omega}$ and:

  \begin{enumerate}
    \item The model of size $\lambda$ is $\LS (\K)^{++}$-saturated.
    \item $\K$ is $(\LS (\K), <\LS (\K)^{+\omega})$-weakly tame.
  \end{enumerate}

  Then $\Ksatp{\LS (\K)^{+\omega}}$ is categorical in all $\lambda' > \LS (\K)^{+\omega}$. In particular, $\K$ is categorical in all $\lambda' \ge \min (\lambda, \sup_{n < \omega} \hanf{\LS (\K)^{+n}})$.
\end{lem}
\begin{proof}
  By Fact \ref{sym-from-categ}, $\K$ has symmetry in $\LS (\K)$ and $\LS (\K)^+$. By Fact \ref{shvi}, $\K$ is superstable in every $\chi \in [\LS (\K), \lambda)$. By Fact \ref{good-frame-weak-tameness}, there is a good $\LS (\K)^+$-frame $\s$ on $\Ksatp{\LS (\K)^+}$. By repeated applications of Corollary \ref{getting-succp}, $\s$ is $\omega$-$\succp$. By Claim \ref{claim-xxx}, $\Ksatp{\LS (\K)^{+\omega}}$ is categorical in all $\lambda' > \LS (\K)^{+\omega}$, and hence by Fact \ref{omitting-type}, $\K$ is categorical in every $\lambda' \ge \sup_{n < \omega} \hanf{\LS (\K)^{+\omega}}$. In particular by Fact \ref{shvi}, $\K$ is stable in $\lambda$, so the model of size $\lambda$ is saturated (hence $\LS (\K)^{+\omega}$-saturated), and so $\K$ must be categorical in all $\lambda' \ge \min (\lambda, \sup_{n < \omega} \hanf{\LS (\K)^{+\omega}})$.
\end{proof}

\begin{cor}\label{abstract-thm-3-proof}
  Assume $\WGCH ([\LS (\K), \LS (\K)^{+\omega}))$ and Claim \ref{claim-xxx}. Assume that $\K$ has amalgamation, arbitrarily large models, and is $\LS (\K)$-tame. If $\K$ is categorical in a $\lambda > \LS (\K)$, then there exists $\chi < H_1$ such that $\K$ is categorical in all $\lambda' \ge \min (\lambda, \chi)$.
\end{cor}
\begin{proof}
  By Fact \ref{jep-decomp}, without loss of generality $\K$ has no maximal models. By Theorem \ref{omitting-categ-transfer} we can assume without loss of generality that $\lambda \ge H_1$ (then we can use Corollary \ref{main-cor} to transfer categoricity downward). By Fact \ref{shvi}, $\K$ is $\LS (\K)$-superstable. By\footnote{In our case, we only use stability in $\lambda$, which follows by \cite[Theorem 5.6]{ss-tame-jsl}.} Fact \ref{tame-sym-ss}, $\K$ is stable in $\lambda$, hence the model of size $\lambda$ is saturated. By Lemma \ref{key-lem-categ-limit}, $\K$ is categorical in all $\lambda' \ge \hanf{\LS (\K)^{+\omega}}$. In particular, $\K$ is categorical in $\left(\hanf{\LS (\K)^{+\omega}}\right)^+$. By Corollary \ref{main-thm}, there exists $\chi < H_1$ such that $\K$ is categorical in all $\lambda' \ge \chi$.
\end{proof}

Without tameness, we can obtain the hypotheses of Lemma \ref{key-lem-categ-limit} from categoricity in a high-enough cardinal. More precisely, to obtain enough weak tameness, we use Theorem \ref{technical-lem}. To obtain the first condition in Lemma \ref{key-lem-categ-limit} (i.e.\ that the model in the categoricity cardinal is sufficiently saturated), we will use\footnote{Shelah claims \cite[Claim IV.7.8]{shelahaecbook} a slightly different result using PCF theory and the existence of certain linear orders: under amalgamation and no maximal models, for $\mu > \LS (\K)$, categoricity in some $\lambda$ so that $\lambda \ge \aleph_{\mu^{+4}} + 2^{2^{\mu}}$ implies that the model of size $\lambda$ is $\mu$-saturated.} Fact \ref{sym-from-categ}.

This allows us to give a proof of \cite[Theorem IV.7.12]{shelahaecbook}. Note that, while we give a slightly different proof to attempt to convince doubters, the result is due to Shelah. In fact, Shelah assumes amalgamation more locally but we haven't fully verified his general proof. As explained in the introduction, we avoid relying on PCF theory or on Shelah's construction of certain linear orders in \cite[Sections IV.5, IV.6]{shelahaecbook}.

\begin{fact}\label{shelah-main-thm}
  Assume Claim \ref{claim-xxx}. Assume that $\K$ has amalgamation. Let $\lambda$ and $\mu$ be cardinals such that:

  \begin{enumerate}
    \item $\K$ is categorical in $\lambda$.
    \item $\mu$ is a limit cardinal with $\cf{\mu} > \LS (\K)$.
    \item For all $\chi < \mu$, $\hanf{\chi} < \lambda$.
    \item For unboundedly many $\chi < \mu$, $\WGCH ([\chi, \chi^{+\omega}))$.
  \end{enumerate}

  Then there exists $\mu_\ast < \mu$ such that $\K$ is categorical in all $\lambda' \ge \hanf{\mu_\ast}$.
\end{fact}
\begin{proof}
  By Fact \ref{jep-decomp}, without loss of generality $\K$ has no maximal models. By Fact \ref{sym-from-categ} (used with $\K$ there standing for $\K_{\ge \chi^+}$ here, for each $\chi < \mu$), the model of size $\lambda$ is $\mu$-saturated. By Theorem \ref{technical-lem}, there exists $\chi < \mu$ so that $\K$ is $(\chi, <\mu)$-weakly tame. Increasing $\chi$ if necessary, assume without loss of generality that $\WGCH ((\chi, \chi^{+\omega}])$ holds. Now apply Lemma \ref{key-lem-categ-limit} with $\K$ there standing for $\K_{\ge \chi}$ here. We get that $\K$ is categorical in all $\lambda' \ge \hanf{\chi^{+\omega}}$, so we obtain the desired conclusion with $\mu_{\ast} := \chi^{+\omega}$.
\end{proof}
\begin{remark}\label{ap-rmk}
  The proof of Fact \ref{shelah-main-thm} given above goes through assuming only that $\K$ has amalgamation below the categoricity cardinal $\lambda$ (using the moreover part of Claim \ref{claim-xxx} to check uniqueness of saturated models).
\end{remark}

\begin{cor}\label{shelah-ap}
  Assume Claim \ref{claim-xxx} and $\WGCH$. If $\K$ has amalgamation and is categorical in \emph{some} $\lambda \ge \hanf{\aleph_{\LS (\K)^+}}$, then $\K$ is categorical in \emph{all} $\lambda' \ge \hanf{\aleph_{\LS (\K)^+}}$.
\end{cor}
\begin{proof}
  Set $\mu := \aleph_{\LS (\K)^+}$ in Fact \ref{shelah-main-thm}.
\end{proof}

We can also state a version using large cardinals instead of amalgamation. This is implicit in Shelah's work (see the remark after \cite[Theorem IV.7.12]{shelahaecbook}), but to the best of our knowledge, the details have not appeared in print before. We will use the following fact, which follows from \cite{kosh362, sh472}. Note that while the results there are stated when $\K$ is the class of models of an $\Ll_{\kappa, \omega}$-theory, Boney observed that the proofs go through just as well in an AEC $\K$ with $\kappa > \LS (\K)$, see the discussion around \cite[Theorem 7.6]{tamelc-jsl}.

\begin{fact}\label{measurable-fact}
  Let $\K$ be an AEC and let $\kappa > \LS (\K)$ be a measurable cardinal. Let $\lambda \ge \hanf{\kappa}$ be such that $\K$ is categorical in $\lambda$. Then:

  \begin{enumerate}
  \item\label{meas-fact-1} \cite{kosh362} $\K_{[\kappa, \lambda)}$ has amalgamation and no maximal models.
  \item\label{meas-fact-2} \cite[Claim 1.16]{sh472} The model of size $\lambda$ is saturated.
  \item\label{meas-fact-3} \cite[Corollary 3.7]{sh472} $\K$ is $(\kappa, <\lambda)$-tame.
  \item\label{meas-fact-4} \cite[Theorem 3.16]{sh472} If $\lambda$ is a successor cardinal, then $\K$ is categorical in all $\lambda' \ge \hanf{\kappa}$.
  \end{enumerate}
\end{fact}

\begin{cor}\label{categ-measurable}
  Assume Claim \ref{claim-xxx} and $\WGCH$. Let $\kappa > \LS (\K)$ be a measurable cardinal. If $\K$ is categorical in \emph{some} $\lambda \ge \hanf{\kappa}$, then $\K$ is categorical in \emph{all}\footnote{The proof gives that there exists $\chi < \hanf{\kappa}$ such that $\K$ is categorical in all $\lambda' \ge \chi$.} $\lambda' \ge \hanf{\kappa}$
\end{cor}
\begin{proof}
  By Fact \ref{measurable-fact}.(\ref{meas-fact-1}), $\K_{[\kappa, \lambda)}$ has amalgamation (and no maximal models, by taking ultrapowers). Note that by Remark \ref{ap-rmk} we do not need amalgamation in $\K_{\ge \lambda}$. By Fact \ref{measurable-fact}.(\ref{meas-fact-2}, the model of size $\lambda$ is saturated. Let $\mu := \aleph_{\kappa^+}$. By Theorem \ref{technical-lem}, there exists $\chi < \mu$ so that $\K$ is $(\chi, <\mu)$-weakly tame. By Lemma \ref{key-lem-categ-limit} (with $\K$ there standing for $\K_{\ge \chi}$ here), $\K$ is categorical in all $\lambda' \ge \hanf{\mu}$. In particular, $\K$ is categorical in $(\hanf{\mu})^+$. By Fact \ref{measurable-fact}.(\ref{meas-fact-4}) (or by Corollary \ref{main-thm}, since by Fact \ref{measurable-fact}.(\ref{meas-fact-3}) $\K$ has enough tameness), $\K$ is also categorical in all $\lambda' \in [\hanf{\kappa}, \hanf{\mu})$.
\end{proof}

The same proof gives:

\begin{cor}
  Assume Claim \ref{claim-xxx} and $\WGCH$. Let $\kappa$ be a measurable cardinal and let $T$ be a theory in $\Ll_{\kappa, \omega}$. If $T$ is categorical in \emph{some} $\lambda \ge \hanf{|T| + \kappa}$, then $T$ is categorical in \emph{all} $\lambda' \ge \hanf{|T| + \kappa}$.
\end{cor}

\section{Summary}\label{summary-sec}

Table \ref{summary-table} summarizes several known approximations of Shelah's eventual categoricity conjecture, for a fixed AEC $\K$. The topmost line and leftmost column contain properties that are either model-theoretic, set-theoretic, or about the categoricity cardinal. The intersection of a line and a column gives a known categoricity transfer for a class having these properties. ``AP'' stands for ``$\K$ has the amalgamation property'', ``Primes'' is short for ``$\K$ has primes'' (Definition \ref{prime-def}), ``s.c.'' is short for ``strongly compact'', and $(\ast)_{\K}$ is the statement ``$\LS (\K) = \kappa$, $\K$ has amalgamation, and $\K$ is $\LS (\K)$-tame''.

Each transfer is described by its type, a comma, and a threshold $\mu$. A ``Full'' type means that categoricity in \emph{some} $\lambda \ge \mu$ implies categoricity in \emph{all} $\lambda' \ge \mu$. A ``Down'' type means that we only know a downward transfer: categoricity in some $\lambda \ge \mu$ implies categoricity in all $\lambda' \in [\mu, \lambda]$ (in this case, we can still do an argument similar to the existence of Hanf numbers \cite{hanf-number} to deduce Shelah's \emph{eventual} categoricity conjecture, see \cite[Conclusion 15.13]{baldwinbook09}). A ``Partial'' type means that we only know that categoricity in some $\lambda \ge \mu$ implies categoricity in \emph{some} $\lambda'$ with $\lambda' \neq \lambda$ (we do \emph{not} require that $\lambda' \ge \mu$). When reading the line beginning with ``categ.\ in a successor'', one should assume that the starting categoricity cardinal $\lambda$ is a successor.

For example, the first entry says that if an AEC $\K$ satisfies the amalgamation property and WGCH together with Claim \ref{claim-xxx} hold, then categoricity in some $\lambda \ge \hanf{\aleph_{\LS (\K)^+}}$ implies categoricity in all $\lambda' \ge \hanf{\aleph_{\LS (\K)^+}}$.

Note that in the first column, it is enough to assume amalgamation below the categoricity cardinal (see Remark \ref{ap-rmk}). So by Fact \ref{measurable-fact}.(\ref{meas-fact-1}) and because strongly compact cardinals are measurable, we can see the properties in the topmost row as being arranged in increasing order of strength. Moreover, the existence of a strongly compact cardinal implies $(\ast)$: amalgamation follows from the methods of \cite[Proposition 1.13]{makkaishelah} and tameness from the main theorem of \cite{tamelc-jsl}.

\begin{fact}\label{star-fact}
  Let $\K$ be an AEC and let $\kappa > \LS (\K)$ be a strongly compact cardinal. Let $\lambda \ge \hanf{\kappa}$. If $\K$ is categorical in $\lambda$, then $(\ast)_{\K_{\ge \kappa}}$ holds.
\end{fact}
\begin{remark}
  An analog of $(\ast)_{\K}$ in the case $\kappa$ is measurable would be given by conclusions (\ref{meas-fact-1})-(\ref{meas-fact-3}) in Fact \ref{measurable-fact}.
\end{remark}

\begin{table}[h]
\begin{center}
  \begin{tabular}{| l | c | c | c |}
    \hline
    & AP & $\kappa > \LS (\K)$ measurable & $\kappa > \LS (\K)$ s.c.\ or $(\ast)_{\K}$ \\ \hline
    
    WGCH and \ref{claim-xxx} & Full, $\hanf{\aleph_{\LS (\K)^+}}$ & Full, $\hanf{\kappa}$ & Full, $\hanf{\kappa}$ \\ \hline
    Categ.\ in a successor & Down, $\beth_{H_1}$ & Down, $\hanf{\kappa}$ & Full, $\hanf{\kappa}$ \\ \hline
    Primes & Partial, $\beth_{\beth_{H_1}}$ & Down, $\hanf{\kappa}$ & Full, $\hanf{\kappa}$ \\ \hline
    No extra hypothesis & Partial, $\beth_{\beth_{H_1}}$ & Partial, $\hanf{\kappa}^+$ & Partial, $\hanf{\kappa}$ \\ 
    \hline
  \end{tabular}
\end{center}
  \caption{Some approximations to Shelah's categoricity conjecture. Properties in the top row are consequences of large cardinal axioms while properties in the first column do not follow (or are not known to follow) from large cardinals. Each entry gives a type of transfer (full, down, or partial) as well as a cardinal threshold. See the beginning of this section for more information on how to read the table.}\label{summary-table}
\end{table}

The results in the first row are Corollary \ref{shelah-ap} and Corollary \ref{categ-measurable} (for the strongly compact case, recall that the properties in the topmost row are in increasing order of strength). The first result in the second row is the downward transfer of \cite{sh394} (see also Corollary \ref{sh394-alternate} for an alternate proof). The second is Fact \ref{measurable-fact}.(\ref{meas-fact-4}). The third is given by Corollary \ref{main-cor} (recalling Fact \ref{star-fact}).

The last two results in the first column are given by Corollary \ref{downward-transfer} (categoricity above $\beth_{\beth_{H_1}}$ implies categoricity in $\beth_{H_1}$). As for the last column, the third result is by Corollary \ref{improved-prime-categ}, and the fourth is by Theorem \ref{omitting-categ-transfer}. Very similar proofs (using Fact \ref{measurable-fact} to deduce the needed amount of amalgamation and tameness) give the corresponding results in the second column. 

\appendix

\section{Shrinking good frames}\label{shrinking-appendix}

We state a generalization of Theorem \ref{good-categ-transfer} to frames that are only defined over classes of saturated models (Shelah studies these frames in more details in \cite{sh842}). This allows us to replace the assumption of tameness by only weak tameness in several results (see Appendix \ref{more-weak-tameness}).

We start by giving a precise definition of these frames (we call them shrinking frames for reasons that will soon become apparent). 

\begin{defin}[Shrinking frame]\label{shrinking-frame-def}
  Let $\lambda$ be an infinite cardinal and let $\theta > \lambda$ be a cardinal or $\infty$. Let $\F := [\lambda, \theta)$ and let $\K$ be an AEC.
    
    We say that $\seq{\s_\mu : \mu \in \F}$ is a \emph{shrinking type-full good $\F$-frame on $\K_{\F}$} (or on $\K$) if:

    \begin{enumerate}
    \item $\K$ is $(\lambda, <\theta)$-weakly tame.
    \item $\s_{\lambda}$ has underlying class $\K_\lambda$.
    \item For each $\mu \in \F$, $\s_{\mu}$ is a type-full good $\mu$-frame with $\K_{\s_{\mu}} = \Ksatp{\mu}_{\mu}$. In particular, $\K$ is categorical in $\lambda$.
    \end{enumerate}
\end{defin}

The reason for the name \emph{shrinking} is that if $\mu < \mu'$ are in $\F$, then the AEC generated by $\K_{\s_{\mu'}}$ is $\Ksatp{\mu'}$, but the underlying class $\K_{\s_{\mu}}$ is only $\Ksatp{\mu}$ which could be a proper subclass of $\Ksatp{\mu'}$ (if $\K$ is not categorical in $\mu'$). Note that that a type-full good $[\lambda, \theta)$-frame (which is categorical in $\lambda$) induces a shrinking frame in a natural way.

  \begin{prop}
    If $\s$ is a type-full good $[\lambda, \theta)$-frame and $\K_{\s}$ is categorical in $\lambda$, then $\seq{\s \rest \Ksatp{\mu} : \mu \in [\lambda, \theta)}$ is a shrinking type-full good $[\lambda, \theta)$-frame.
  \end{prop}
  \begin{proof}
    Straightforward, recalling Fact \ref{satfact}.
  \end{proof}

  Shrinking frames can be built using Fact \ref{good-frame-weak-tameness}:

\begin{thm}\label{shrinking-frame-weak-tameness}
  Let $\K$ be an AEC. Let $\lambda > \LS (\K)$. Assume that for every $\mu \in [\LS (\K), \lambda)$, $\K$ is $\mu$-superstable and has $\mu$-symmetry.

    If $\K$ is $(\LS (\K), <\lambda)$-weakly tame, then there exists a shrinking type-full good $[\LS (\K)^+, \lambda)$-frame on $\K$.
\end{thm}
\begin{proof}
  Let $\F := [\LS (\K)^+, \lambda)$. By Fact \ref{good-frame-weak-tameness}, for each $\mu \in \F$, there exists a type-full good $\mu$-frame on $\Ksatp{\mu}_{\mu}$. The result follows.
\end{proof}

We now study how forking in two different cardinals interact in a shrinking frame. The following notion is key:

\begin{defin}\label{compat-def}
  Let $\K$ be an AEC. Let $\LS (\K) \le \lambda < \mu$. Let $\s_\lambda$ be a type-full good $\lambda$-frame on $\Ksatp{\lambda}_\lambda$ and $\s_{\mu}$ be a type-full good $\mu$-frame on $\Ksatp{\mu}_\mu$. We say that $\s_\lambda$ and $\s_\mu$ are \emph{compatible} if for any $M \lea N$ in $\K_{\s_{\mu}}$, and $p \in \gS (N)$, $p$ does not $\s_{\mu}$-fork over $M$ if and only if there exists $M_0 \lea M$ with $M_0 \in \K_{\s_{\lambda}}$ so that $p \rest N_0$ does not $\s_{\lambda}$-fork over $M_0$ for every $N_0 \in \K_{\s_{\lambda}}$ with $M_0 \lea N_0 \lea N$.
\end{defin}

Intuitively, compatibility says that that forking in $\s_{\mu}$ can be computed using forking in $\s_{\lambda}$. In fact, it can be described in a canonical way (i.e.\ using Shelah's description of the extended frame, see \cite[Section II.2]{shelahaecbook}). The following result is key:

\begin{thm}\label{shrinking-compat}
  Let $\seq{\s_\mu : \mu \in \F}$ be a shrinking type-full good $\F$-frame on the AEC $\K$. Let $\lambda < \mu$ be in $\F$. Then $\s_{\lambda}$ and $\s_{\mu}$ are compatible.
\end{thm}

For the proof, we will use the following result which gives an explicit description of forking in any categorical good frame:

\begin{fact}[The canonicity theorem, 9.6 in \cite{indep-aec-apal}]\label{canon-fact}
  Let $\s$ be a type-full good $\lambda$-frame with underlying class $\K_\lambda$.
  If $M \lea N$ are limit models in $\K_{\lambda}$, then for any $p \in \gS (N)$, $p$ does not $\s$-fork over $M$ if and only if there exists $M' \in \K_{\lambda}$ such that $M$ is limit over $M'$ and $p$ does not $\lambda$-split over $M'$.
\end{fact}

\begin{proof}[Proof of Theorem \ref{shrinking-compat}]
  Note that by uniqueness of limit models, every model in $\K_{\s_{\mu}}$ is limit.

  For $M, N \in \K_{\s_{\mu}}$ with $M \lea N$, let us say that $p \in \gS (N)$ \emph{does not ($\ge \s_{\lambda})$-fork over $M$} if it satisfies the condition in Definition \ref{compat-def}, namely there exists $M_0 \lea M$ with $M_0 \in \K_{\s_{\lambda}}$ so that $p \rest N_0$ does not $\s_{\lambda}$-fork over $M_0$ for every $N_0 \in \K_{\s_{\lambda}}$ with $M_0 \lea N_0 \lea N$. Let us say that $p$ \emph{does not $\mu$-fork over $M$} if it satisfies the description of the canonicity theorem, namely there exists $M' \in \K_{\s_{\mu}}$ such that $M$ is limit over $M'$ and $p$ does not $\mu$-split over $M'$. Notice that by the canonicity theorem (Fact \ref{canon-fact}), $p$ does not $\s_{\mu}$-fork over $M$ if and only if $p$ does not $\mu$-fork over $M$. Thus it is enough to show that $p$ does not $(\ge \s_{\lambda})$-fork over $M$ if and only if $p$ does not $\mu$-fork over $M$. We first show one direction:
  
  \paragraph{\underline{Claim}} Let $M \lea N$ both be in $\K_{\s_{\mu}}$ and let $p \in \gS (N)$. If $p$ does not $(\ge \s_{\lambda})$-fork over $M$, then $p$ does not $\mu$-fork over $M$.
  
  \paragraph{\underline{Proof of Claim}} We know that $M$ is limit, so let $\seq{M_i : i < \delta}$ witness it, i.e.\ $\delta$ is limit, for all $i < \delta$, $M_i \in \K_{\s_{\mu}}$, $M_{i + 1}$ is universal over $M_i$, and $\bigcup_{i < \delta} M_i = M$. By \cite[Claim II.2.11.(5)]{shelahaecbook}, there exists $i < \delta$ such that $p \rest M$ does not $(\ge \s_{\lambda})$-fork over $M_i$. By \cite[Claim II.2.11.(4)]{shelahaecbook}, $p$ does not $(\ge \s_{\lambda})$-fork over $M_i$. By weak tameness and the uniqueness property of $\s$, $(\ge \s_{\lambda})$-forking has the uniqueness property (see the proof of \cite[Theorem 3.2]{ext-frame-jml}). By \cite[Lemma 4.2]{bgkv-apal}, $(\ge \s_{\lambda})$-nonforking must be extended by $\mu$-nonsplitting, so $p$ does not $\mu$-split over $M_i$. Therefore $p$ does not $\mu$-fork over $M$, as desired. $\dagger_{\text{Claim}}$.

  Now as observed above, $(\ge \s_{\lambda})$-forking has the uniqueness property. Also, $\mu$-forking has the extension property (as $\s_{\mu}$-forking has it). The claim tells us that $\mu$-nonforking extends $(\ge \s_{\lambda})$-forking and hence by \cite[Lemma 4.1]{bgkv-apal}, they are the same.
\end{proof}

Thus we can define a global notion of forking inside the frame:

\begin{defin}\label{shrinking-forking}
  Assume that $\seq{\s_\mu : \mu \in \F}$ is a shrinking type-full good $[\lambda, \theta)$-frame. Let $\mu \le \mu'$ be in $\F$ and let $M \lea N$ be such that $M \in \K_{\s_{\mu}}$ and $M' \in \K_{\s_{\mu'}}$. Let $p \in \gS (N)$. We say that \emph{$p$ does not fork over $M$} if there exists $M_0 \lea M$ so that $M_0 \in \K_{\s_{\lambda}}$ and for every $N_0 \in \K_{\s_{\lambda}}$ with $M_0 \lea N_0 \lea N$, $p \rest N_0$ does not $\s_{\lambda}$-fork over $M_0$.
\end{defin}

\begin{thm}\label{shrinking-forking-props}
  Assume that $\seq{\s_\mu : \mu \in \F}$ is a shrinking type-full good $[\lambda, \theta)$-frame. Then forking (as defined in Definition \ref{shrinking-forking}) has the usual properties: invariance, monotonicity, extension, uniqueness, transitivity, local character, and symmetry.
\end{thm}
\begin{proof}[Proof sketch]
  Invariance, monotonicity, transitivity, and local character are straightforward. Symmetry is also straightforward (once we have it when the domain and the base have the same size, it is a simple use of monotonicity). Uniqueness is by weak tameness, and extension is as in \cite[Proposition 5.1]{vv-symmetry-transfer-v3}.
\end{proof}

We can now state a generalization of Theorem \ref{good-categ-transfer} and sketch a proof:

\begin{thm}\label{good-categ-transfer-weak}
  Let $\seq{\s_{\mu} : \mu \in [\lambda, \theta)}$ be a shrinking type-full good $[\lambda, \theta)$-frame on $\K$. Let $\mu \in [\lambda, \theta)$. If $\Ksatp{\mu}$ is categorical in $\mu^+$, then $\K$ is categorical in every $\mu \in [\lambda, \theta]$.
\end{thm}
\begin{proof}[Proof sketch]
  First note that in the upward transfer of Grossberg and VanDieren (Fact \ref{upward-transfer-2}), it is implicit that tameness can be weakened to weak tameness (Remark \ref{gv-upward-rmk}). The rest of the proof of Theorem \ref{good-categ-transfer} (the downward part) is as before: we use Theorem \ref{shrinking-forking-props} and make sure that anytime a resolution is taken, all the components are saturated.
\end{proof}

\section{More on weak tameness}\label{more-weak-tameness}

We use Theorem \ref{good-categ-transfer-weak} to replace tameness by weak tameness in some of the results of the second part of this paper. Everywhere in this section, we assume:

\begin{hypothesis}
  $\K$ is an AEC with amalgamation.
\end{hypothesis}

First, we state a stronger version of the main lemma (Lemma \ref{main-lem}):

\begin{lem}\label{main-lem-2}
  Assume that $\K$ has no maximal models. Let $\theta \ge \lambda > \LS (\K)^+$ be such that $\K$ is $(\LS (\K), <\theta)$-weakly tame. Assume that $\lambda$ is a successor cardinal. If $\K$ is categorical in $\lambda$, then $\Ksatp{\LS (\K)^+}$ is categorical in all $\mu \in [\LS (\K)^+, \theta]$.
\end{lem}
\begin{proof}
  By the upward transfer of Grossberg and VanDieren (Fact \ref{gv-upward}), $\K$ (and therefore $\Ksatp{\LS (\K)^+})$ is categorical in every $\lambda' \in [\lambda, \theta]$. It remains to show the downward part. By Fact \ref{shvi}, $\K$ is superstable in every $\mu \in [\LS (\K), \lambda)$. Since $\lambda$ is a successor, the model of size $\lambda$ is saturated. By Fact \ref{sym-from-categ}, $\K$ has symmetry in every $\mu \in [\LS (\K), \lambda)$. By Theorem \ref{shrinking-frame-weak-tameness}, there is a shrinking type-full good $[\LS (\K)^+, \lambda)$-frame on $\K$. By Theorem \ref{good-categ-transfer-weak}, $\Ksatp{\LS (\K)^+}$ is categorical in all $\lambda' \in [\LS (\K)^+, \lambda]$. 
\end{proof}

We can improve on Corollary \ref{main-cor}:

\begin{cor}\label{main-cor-2}
  Assume that $\K$ has arbitrarily large models. Let $\LS (\K) < \lambda_0 < \lambda$. Assume that $\K$ is $(\LS (\K), <\lambda)$-weakly tame. If $\lambda$ is a successor cardinal and $\K$ is categorical in $\lambda_0$ and $\lambda$, then $\K$ is categorical in all $\lambda' \in (\lambda_0, \lambda)$.
\end{cor}
\begin{proof}
  As in the proof of Corollary \ref{main-cor}, using Lemma \ref{main-lem-2} (with $\lambda, \theta$ there standing for $\lambda, \lambda$ here).
\end{proof}

Corollary \ref{main-thm} can similarly be generalized:

\begin{cor}\label{main-thm-2}
  Let $\lambda \ge H_1$ be a successor cardinal and assume that $\K$ is $(\LS (\K), <\lambda)$-weakly tame. If $\K$ is categorical in $\lambda$, then there exists $\chi < H_1$ such that $\K$ is categorical in all $\lambda' \in [\chi, \lambda)$.
\end{cor}
\begin{proof}
  As in the proof of Corollary \ref{main-thm}, using Lemma \ref{main-lem-2} (with $\lambda, \theta$ there standing for $\lambda, \lambda$ here).
\end{proof}
\begin{remark}
  It is unclear how to generalize the results using primes: the proof of Fact \ref{prime-fact} uses tameness (for all models) heavily, and we do not know how to generalize it to weakly tame AECs.
\end{remark}

We can use Corollary \ref{main-thm-2} to give an alternate proof to the main theorem of \cite{sh394}.

\begin{cor}\label{sh394-alternate}
  If $\K$ is categorical in \emph{some successor} $\lambda \ge \beth_{H_1}$, then there exists $\mu < \beth_{H_1}$ such that $\K$ is categorical in \emph{all} $\lambda' \in [\mu, \lambda)$.
\end{cor}
\begin{proof}
  Without loss of generality (Fact \ref{jep-decomp}), $\K$ has no maximal models. By Fact \ref{shvi}, $\K$ is stable below $\lambda$, so the model of size $\lambda$ is saturated. By Theorem \ref{technical-lem}, there exists $\chi < H_1$ such that $\K$ is $(\chi, <\lambda)$-weakly tame. By Corollary \ref{main-thm-2} (applied to $\K_{\ge \chi}$), there exists $\mu < \hanf{\chi} < \beth_{H_1}$ such that $\K$ is categorical in all $\lambda' \in [\mu, \lambda)$.
\end{proof}

Generalizing Corollary \ref{abstract-thm-3-proof} is harder. The problem is how to ensure that the model in the categoricity cardinal has enough saturation. We give a consistency result in case $\lambda \ge H_1$.

\begin{cor}
  Assume that $2^{\LS (\K)} = 2^{\LS (\K)^+}$, $\WGCH ([\LS (\K)^+, \LS (\K)^{+\omega}))$, and Claim \ref{claim-xxx} holds. Assume that $\K$ is $(\LS (\K), <H_1)$-weakly tame. If $\K$ is categorical in \emph{some} $\lambda \ge H_1$, then there exists $\chi < H_1$ such that $\K$ is categorical in \emph{all} $\lambda' \ge \chi$.
\end{cor}
\begin{proof}
  Without loss of generality (Fact \ref{jep-decomp}), $\K$ has no maximal models. By Fact \ref{shvi}, $\K$ is superstable in every $\mu \in [\LS (\K), \lambda)$. Since $2^{\LS (\K)} = 2^{\LS (\K)^+}$, $H_1 = \hanf{\LS (\K)^+}$, so by Fact \ref{sym-from-categ}, $\K$ has symmetry in $\LS (\K)^+$. By Fact \ref{sym-unionsat}, $\Ksatp{\LS (\K)^+}$ is an AEC with $\LS (\Ksatp{\LS (\K)^+}) = \LS (\K)^+$. In particular, the model of size $\lambda$ is $\LS (\K)^+$-saturated. By Fact \ref{good-frame-weak-tameness}, there exists a type-full good $\LS (\K)^+$-frame $\s$ on $\Ksatp{\LS (\K)^+}_{\LS (\K)^+}$. By iterating Corollary \ref{getting-succp}, $\s$ is $\omega$-$\succp$. As in the proof of Corollary \ref{abstract-thm-3-proof}, we get that $\K$ is categorical on a tail of cardinals. By Theorem \ref{technical-lem}, $\K$ is $\chi$-weakly tame, so combining this with the hypothesis of $(\LS (\K), <H_1)$-tameness, $\K$ is $\LS (\K)$-weakly tame. Now apply Corollary \ref{main-thm-2}.
\end{proof}

\section{Superstability for long types}

We generalize Definition \ref{ss assm} to types of more than one element and use it to prove an extension property for $1$-forking (recall Definition \ref{1-forking-def}). This is used to give a converse to Lemma \ref{perp-nf-1} in the next appendix (but is not needed for the main body of this paper). Everywhere below, $\K$ is an AEC.

\begin{defin}\label{ss-parametrized}
  Let $\alpha \le \omega$ be a cardinal. $\K$ is \emph{$(<\alpha, \mu)$-superstable} (or \emph{$(<\alpha)$-superstable in $\mu$}) if it satisfies Definition \ref{ss-parametrized} except that in addition in condition (\ref{split assm}) there we allow $p \in \gS^{<\alpha} (M_\delta)$ rather than just $p \in \gS (M_\delta)$ (that is, $p$ need not have length one). \emph{$(\le \alpha, \mu)$-superstable} means $(<\alpha^+, \mu)$-superstable. When $\alpha = 2$, we omit it (that is, $\mu$-superstable means $(< 2, \mu)$-superstable which is the same as $(\le 1, \mu)$-superstable).
\end{defin}

While not formally equivalent (although we do not know of any examples separating the two),  $\mu$-superstability and $(<\omega, \mu)$-superstability are very close. For example, the proof of Fact \ref{shvi} also gives:

\begin{fact}\label{shvi2} 
  Let $\mu \ge \LS (\K)$. If $\K$ has amalgamation, no maximal models, and is categorical in a $\lambda > \mu$, then $\K$ is $(<\omega, \mu)$-superstable.
\end{fact}

Even without categoricity, we can obtain eventual $(<\omega)$-superstability from just $(\le 1)$-superstability and tameness. This uses another equivalent definition of superstability: solvability:

\begin{thm}\label{tame-long-ss}
  Assume $\K$ has amalgamation, no maximal models, and is $\LS (\K)$-tame. If $\K$ is $\LS (\K)$-superstable, then there exists $\mu_0 < H_1$ such that $\K$ is $(<\omega)$-superstable in every $\mu \ge \mu_0$.
\end{thm}
\begin{proof}[Proof sketch]
  By \cite[Corollary 5.10]{gv-superstability-v3}, there exists $\mu_0 < H_1$ such that $\K$ is $(\mu_0, \mu)$-solvable for every $\mu \ge \mu_0$. This means \cite[Definition IV.1.4.(1)]{shelahaecbook} that for every $\mu \ge \mu_0$, there exists an EM Blueprint $\Phi$ so that $\text{EM}_{\tau (\K)} (I, \Phi)$ is a superlimit in $\K$ for every linear order $I$ of size $\mu$. Intuitively, it gives a weak version of categoricity in $\mu$. As observed in \cite[Remark 6.2]{gv-superstability-v3}, this weak version is enough for the proof of the Shelah-Villaveces theorem to go through, hence by Fact \ref{shvi2}, $\K$ is $(<\omega)$-superstable in $\mu$ for every $\mu \ge \mu_0$.
\end{proof}
\begin{remark}
  If $\K$ has amalgamation, is $\LS (\K)$-tame for types of length less than $\omega$, and is $(<\omega, \LS (\K))$-superstable, then (by the proof of \cite[Proposition 10.10]{indep-aec-apal}) $\K$ is $(<\omega)$-superstable in every $\mu \ge \LS (\K)$. However here we want to stick to using regular tameness (i.e.\ tameness for types of length one).
\end{remark}

To prove the extension property for $1$-forking, we will use:

\begin{fact}[Extension property for splitting, Proposition 5.1 in \cite{vv-symmetry-transfer-v3}]\label{ext-splitting}
  Let $\K$ be an AEC, $\theta > \LS (\K)$. Let $\alpha \le \omega$ be a cardinal and assume that $\K$ is $(<\alpha)$-superstable in every $\mu \in [\LS (\K), \theta)$. Let $M_0 \lea M \lea N$ be in $\K_{[\LS (\K), \theta)}$, with $M$ limit over $M_0$. Let $p \in \gS^{<\alpha} (M)$ be such that $p$ does not $\LS (\K)$-split over $M_0$. Then there exists an extension $q \in \gS^{<\alpha} (N)$ of $p$ which does not $\LS (\K)$-split over $M_0$.
\end{fact}

\begin{thm}\label{piecewise-ext}
  Let $\theta > \LS (\K)$. Write $\F := [\LS (\K), \theta)$. Let $\s$ be a type-full good $\F$-frame with underlying class $\K_{\F}$. Let $\alpha \le \omega$ be a cardinal and assume that $\K$ is $(<\alpha, \mu)$-superstable for every $\mu \in \F$. Let $M \lea N$ be in $\K_{\F}$ with $M$ a limit model. Let $p \in \gS^{<\alpha} (M)$. Then there exists $q \in \gS^{<\alpha} (N)$ that extends $p$ so that $q$ does not $1$-$\s$-fork over $M$ (recall Definition \ref{1-forking-def}).
\end{thm}
\begin{proof}
  Without loss of generality, $N$ is a limit model. Let $\mu := \|M\|$. By $(<\alpha)$-superstability, there exists $M_0 \in \K_{\mu}$ such that $M$ is limit over $M_0$ and $p$ does not $\mu$-split over $M_0$. By Fact \ref{ext-splitting}, there exists $q \in \gS (N)$ extending $p$ so that $q$ does not $\mu$-split over $M_0$. We claim that $q$ does not $1$-$\s$-fork over $M$. Let $I \subseteq \ell (p)$ have size one. By monotonicity of splitting, $q^I$ does not $\mu$-split over $M_0$. By local character, let $N_0 \lea N$ be such that $M \lea N_0$, $N_0 \in \K_\mu$, $N_0$ is limit, and $q^I$ does not $\s$-fork over $N_0$. By monotonicity of splitting again, $q^I \rest N_0$ does not $\mu$-split over $M_0$. By the canonicity theorem (Fact \ref{canon-fact}) applied to the frame $\s \rest \K_\mu$, $q^I \rest N_0$ does not $\s$-fork over $M$. By transitivity, $q^I$ does not $\s$-fork over $M$, as desired.
\end{proof}

\section{More on global orthogonality}

Assuming superstability for types of length two, we prove a converse to Lemma \ref{perp-nf-1}, partially answering Question \ref{perp-nf-question}. We then prove a few more facts about global orthogonality and derive an alternative proof of the upward categoricity transfer of Grossberg and VanDieren (Fact \ref{gv-upward}). This material is not needed for the main body of this paper.

\begin{hypothesis}\label{appendix-3-hyp} \
  \begin{enumerate}
    \item $\K$ is an AEC.
    \item $\theta > \LS (\K)$ is a cardinal or $\infty$. We set $\F := [\LS (\K), \theta)$.
    \item $\s = (\K_\F, \nf)$ is a type-full good $\F$-frame.
    \item $\K$ is $(\le 2)$-superstable in every $\mu \in \F$.
  \end{enumerate}
\end{hypothesis}
\begin{remark}
  Compared to Hypothesis \ref{sec-5-hyp}, we have added $(\le 2)$-superstability. Note that this would follow automatically if $\s$ was a type-full good frame for types of length two, hence it is a minor addition. It also holds if $\K$ is categorical above $\F$ (Fact \ref{shvi2}) or even if it is just tame (Theorem \ref{tame-long-ss}).
\end{remark}

\begin{lem}\label{perp-nf-2}
  Let $M_0 \lea M$ be both in $\K_{\F}$ with $M_0 \in \K_{\LS (\K)}$ limit. Let $p, q \in \gS (M)$ be nonalgebraic so that both do not fork over $M_0$. If $p \wkperp q$, then $p \rest M_0 \wkperp q \rest M_0$.
\end{lem}
\begin{proof}
  Assume that $p \rest M_0 \not \wkperp q \rest M_0$. We show that $p \not \wkperp q$. Fix $N \in \K_{\F}$ with $M_0 \lea N$ and let $a, b \in |N|$ realize in $N$ $p \rest M_0$ and $q \rest M_0$ respectively. Assume that $\seq{ab}$ is \emph{not} independent in $(M_0, N)$. Let $r := \gtp (ab / M_0; N)$. By Theorem \ref{piecewise-ext}, there exists $r' \in \gS^2 (M)$ that extends $r$ and so that $r'$ does not $1$-$\s$-fork over $M_0$ (recall Definition \ref{1-forking-def}). Let $N' \gea M$ and let $\seq{a'b'}$ realize $r'$ in $N'$. Then $\gtp (a' / M; N')$ does not fork over $M_0$ and extends $p \rest M_0$, hence $a'$ must realize $p$ in $N'$. Similarly, $b'$ realizes $q$. We claim that $\seq{a'b'}$ is \emph{not} independent in $(M, N')$, hence $p \not \wkperp q$. If $\seq{a'b'}$ were independent in $(M, N')$, there would exist $N'' \gea N'$ and $M' \lea N''$ so that $M \lea M'$, $b \in |M'|$, and $\gtp (a' / M'; N'')$ does not fork over $M$. By transitivity, $\gtp (a' / M'; N'')$ does not fork over $M_0$. This shows that $\seq{a' b'}$ is independent in $(M_0, N'')$, so since $\gtp (a' b' / M_0; N'') = \gtp (a b / M_0; N)$, we must have that $\seq{ab}$ is independent in $(M_0, N)$, a contradiction.
\end{proof}

We obtain:

\begin{thm}\label{perp-global}
  Let $M \in \K_{\F}$ and $p, q \in \gS (M)$. Then:

  \begin{enumerate}
    \item If $M \in \Ksatp{\LS (\K)}_{\F}$, then $p \perp q$ if and only if $p \wkperp q$.
    \item If $M \in \Ksatp{\LS (\K)}_{\F}$, then $p \perp q$ if and only if $q \perp p$.
    \item If $M_0 \in \Ksatp{\LS (\K)}_{\F}$ is such that $M_0 \lea M$ and both $p$ and $q$ do not fork over $M_0$, then $p \perp q$ if and only if $p \rest M_0 \perp q \rest M_0$.
  \end{enumerate}
\end{thm}
\begin{proof} \
  \begin{enumerate}
  \item If $p \perp q$, then $p \wkperp q$ by definition. Conversely, assume that $p \wkperp q$. Fix a limit $M_0 \in \K_{\LS (\K)}$ such that $M_0 \lea M$ and both $p$ and $q$ do not fork over $M_0$. By Lemma \ref{perp-nf-2}, $p \rest M_0 \wkperp q \rest M_0$. By Lemma \ref{perp-fund}.(\ref{perp-fund-1}), $p \rest M_0 \perp q \rest M_0$. By Lemma \ref{perp-nf-1}, $p \perp q$.
  \item A similar proof, using (\ref{perp-fund-2}) instead of (\ref{perp-fund-1}) in Lemma \ref{perp-fund}.
  \item By local character and transitivity, we can fix a limit $M_0' \in \K_{\LS (\K)}$ such that $M_0' \lea M_0$ and both $p$ and $q$ do not fork over $M_0'$. Now by what has been proven above and Lemmas \ref{perp-nf-1} and \ref{perp-nf-2}, $p \perp q$ if and only if $p \rest M_0' \perp q \rest M_0'$ if and only if $p \rest M_0 \perp q \rest M_0$.
  \end{enumerate}
\end{proof}

We can now give another proof of the upward transfer of unidimensionality (the second part of the proof of Theorem \ref{unidim-transfer}). This does not use Fact \ref{upward-transfer-2}.

\begin{lem}\label{unidim-upward}
  Let $\mu < \lambda$ be in $\F$. If $\s$ is $\mu$-unidimensional, then $\s$ is $\lambda$-unidimensional.
\end{lem}
\begin{proof}
  Assume that $\s$ is \emph{not} $\lambda$-unidimensional. Let $M_0 \in \K_{\mu}$ be limit and let $p_0 \in \gS (M_0)$ be minimal. We show that there exists a limit $M_0' \in \K_{\mu}$, $p_0', q_0' \in \gS (M_0')$ such that $p_0'$ extends $p_0$ and $p_0' \perp q_0'$. This will show that $\K$ is not $\mu$-unidimensional by Lemma \ref{technical-multidim-equiv}. Let $M \in \K_\lambda$ be saturated such that $M_0 \lea M$ and let $p \in \gS (M)$ be the nonforking extension of $p_0$. By non-$\lambda$-unidimensionality (and Lemma \ref{technical-multidim-equiv}), there exists $q \in \gS (M)$ so that $p \perp q$. Let $M_0' \in \K_\mu$ be limit such that $M_0 \lea M_0' \lea M$ and $q$ does not fork over $M_0'$. Let $p_0' := p \rest M_0'$, $q_0' := q \rest M_0'$. By Theorem \ref{perp-global}, $p_0' \perp q_0'$, as desired.
\end{proof}

We obtain the promised alternate proof to Grossberg-VanDieren. For this corollary, we drop Hypothesis \ref{appendix-3-hyp}. 

\begin{cor}\label{gv-alternate}
  Let $\K$ be an AEC with amalgamation and arbitrarily large models. If $\K$ is $\LS (\K)$-tame and categorical in a successor $\lambda > \LS (\K)^+$, then $\K$ is categorical in all $\mu \ge \lambda$.
\end{cor}
\begin{proof}
  By Fact \ref{jep-decomp}, we can assume without loss of generality that $\K$ has no maximal models. By Fact \ref{shvi}, $\K$ is $\LS (\K)$-superstable. By Fact \ref{tame-sym-ss}, $\K$ is superstable and has symmetry in every $\mu \ge \LS (\K)$. By Theorem \ref{omitting-categ-transfer}, $\K$ is categorical in a proper class of cardinals. By Fact \ref{shvi2}, $\K$ is $(<\omega)$-superstable in every $\mu \ge \LS (\K)$. Now say $\lambda = \lambda_0^+$. By Facts \ref{good-frame-weak-tameness} and \ref{frame-upward-transfer}, there exists a type-full good $(\ge \lambda_0)$-frame $\s$ with underlying class $\Ksatp{\LS (\K)^+}_{\ge \LS (\K)^+}$. By Fact \ref{satfact}, we can restrict the frame further to have underlying class $\Ksatp{\lambda_0}_{\ge \lambda_0}$. By Corollary \ref{good-categ-transfer} (using Lemma \ref{unidim-upward} to transfer unidimensionality up), $\Ksatp{\lambda_0}$ is categorical in every $\mu \ge \lambda_0$. Now $\Ksatp{\lambda_0}_{\ge \lambda} = \K_{\ge \lambda}$ (by categoricity in $\lambda$), so the result follows.
\end{proof}
\begin{remark}
  Similarly to what was said in Remark \ref{main-cor-rmk}, it is also possible to use this argument to prove that a $\LS (\K)$-tame AEC with amalgamation and arbitrarily large models categorical in $\LS (\K)$ and $\LS (\K)^+$ is categorical everywhere \cite[Theorem 6.3]{tamenessthree}. However we do not know that we have a good frame in $\LS (\K)$, so the proof is more complicated.
\end{remark}
\begin{remark}
  As opposed to Grossberg and VanDieren's proof, our proof of Corollary \ref{gv-alternate} is global: it cannot be turned into an argument for Fact \ref{upward-transfer-2} (e.g.\ if $\theta$ there is $\LS (\K)^{+\omega}$ we cannot use Theorem \ref{omitting-categ-transfer}). It is also not clear how to replace tameness with weak tameness in the proof.
\end{remark}

\section{A proof of Fact \ref{weakly-succ-fact}}\label{weakly-succ-appendix}

\begin{hypothesis}
  $\K$ is an AEC, $\lambda \ge \LS (\K)$.
\end{hypothesis}

We give a full proof of Fact \ref{weakly-succ-fact}. Shelah's proof in $\odot_4$ of the proof of \cite[Theorem IV.7.12]{shelahaecbook} skips some steps (for example it is not clear how we can make sure there that $f_{\eta \smallfrown 0}[M_{\ell (\eta) + 1}] = f_{\eta \smallfrown 1}[M_{\ell (\eta) + 1}]$). The proof we give here is similar in spirit to Shelah's, but we use a stronger blackbox (relying on weak diamond) which avoids having to deal with all of Shelah's renaming steps.

We will use the combinatorial principle $\Theta_{\lambda^+}$ introduced (for $\lambda = \aleph_0$) in \cite{dvsh65}. 

\begin{defin}\label{theta-def}
  $\Theta_{\lambda^+}$ holds if for every $\seq{f_\eta \in \fct{\lambda^+}{\lambda^+} : \eta \in \fct{\lambda^+}{2}}$, there exists $\eta \in \fct{\lambda^+}{2}$ such that the set

  $$S_\eta := \{\delta < \lambda^+ \mid \exists \nu \in \fct{\lambda^+}{2} : f_\eta \rest \delta = f_\nu \rest \delta \land \eta \rest \delta = \nu \rest \delta \land \eta (\delta) \neq \nu (\delta)\}$$

  is stationary.
\end{defin}
\begin{remark}\label{theta-rmk}
  Instead of requiring that $f_\eta \in \fct{\lambda^+}{\lambda^+}$, we can assume only that $f_\eta$ is a partial function from $\lambda^+$ to $\lambda^+$ (we can always extend $f_\eta$ arbitrarily to an actual function).
\end{remark}

\begin{fact}[6.1 in \cite{dvsh65}]\label{theta-fact}
  If $2^{\lambda} < 2^{\lambda^+}$, then $\Theta_{\lambda^+}$ holds.
\end{fact}

In \cite[Claim 1.4.(2)]{sh576}, Shelah shows assuming the weak diamond that given a tree witnessing failure of amalgamation in $\lambda$, there cannot be a universal model of cardinality $\lambda^+$. A similar proof gives a two-dimensional version:

\begin{lem}\label{weak-diamond-lem}
  Assume $2^{\lambda} < 2^{\lambda^+}$. Let $\seq{M_\eta : \eta \in \fct{\le \lambda^+}{2}}$ be a strictly increasing continuous tree with $M_\eta \in \K_\lambda$ for all $\eta \in \fct{<\lambda^+}{2}$. Let $\seq{M_\alpha : \alpha \le \lambda^+}$ be an increasing continuous chain and let $\seq{f_\eta : \eta \in \fct{\le \lambda^+}{2}}$ be such that for any $\eta \in \fct{\le \lambda^+}{2}$, $f_\eta : M_{\ell (\eta)} \rightarrow M_\eta$ and for any $\alpha < \ell (\eta)$, $f_\eta \rest M_{\alpha} = f_{\eta \rest \alpha}$.

  Assume that there exists $N \in \K_{\lambda^+}$ with $M_{\lambda^+} \lea N$ such that for all $\eta \in \fct{\lambda^+}{2}$, there exists $g_\eta : M_\eta \rightarrow N$ such that the following diagram commutes:

  \[
  \xymatrix{
    N  & \\
    M_{\lambda^+} \ar[u] \ar[r]_{f_\eta} & M_{\eta} \ar@{.>}[ul]_{g_\eta} \\
  }
  \]

  Then there exists $\rho \in \fct{<\lambda^+}{2}$, $\eta, \nu \in \fct{\lambda^+}{2}$, $\delta < \lambda^+$ such that $\rho = \eta \rest \delta = \nu \rest \delta$, $0 = \eta (\delta) \neq \nu (\delta) = 1$, and the following diagram commutes:

    
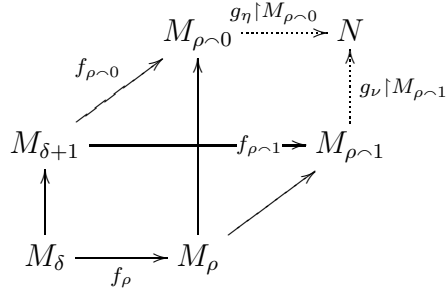
\[
  \xymatrix{ & M_{\rho \smallfrown 0} \ar@{.>}[r]^{g_{\eta} \rest M_{\rho \smallfrown 0}} & N \\
    M_{\delta + 1} \ar[ru]^{f_{\rho \smallfrown 0}} \ar[rr]|>>>>>>{f_{\rho \smallfrown 1}} & & M_{\rho \smallfrown 1} \ar@{.>}[u]_{g_{\nu} \rest M_{\rho \smallfrown 1}} \\
    M_{\delta} \ar[u] \ar[r]_{f_\rho} & M_\rho \ar[uu] \ar[ur] & \\
  }
  \]
\end{lem}
\begin{proof}
  Fix $N \in \K_{\lambda^+}$, $f_\eta$, $g_\eta$ as in the statement of the lemma. By renaming everything, we can assume without loss of generality that $|N| \subseteq \lambda^+$ and $|M_\eta| \subseteq \lambda^+$ for all $\eta \in \fct{\le \lambda^+}{2}$. 

  By Fact \ref{theta-fact}, $\Theta_{\lambda^+}$ holds. We use it with the sequence $\seq{g_\eta : \eta \in \fct{\lambda^+}{2}}$ (see Remark \ref{theta-rmk}). We obtain $\eta \in \fct{\lambda^+}{2}$ such that the set $S_\eta$ of Definition \ref{theta-def} is stationary. Let $C := \{\delta < \lambda^+ \mid |M_{\eta \rest \delta}| \subseteq \delta\}$. Clearly, $C$ is club so let $\delta \in S_\eta \cap C$ and let $\nu \in \fct{\lambda^+}{2}$ be as given by the definition of $S_\eta$ (i.e.\ $g_\eta \rest \delta = f_\nu \rest \delta$, $\eta \rest \delta = \nu \rest \delta$, and without loss of generality $0 = \eta (\delta) \neq \nu (\delta) = 1$). Let $\rho := \eta \rest \delta = \nu \rest \delta$. Then $\rho, \eta, \nu, \delta$ are as desired. The main point is that since $|M_\rho| \subseteq \delta$ (by definition of $\delta$), $g_{\eta} \rest M_{\rho} = g_\nu \rest M_{\rho}$.

\end{proof}

The next technical property is of great importance in Chapter II and III of \cite{shelahaecbook}. The definition below follows \cite[Definition 4.1.5]{jrsh875} (but as usual, we work only with type-full frames).

\begin{defin}\label{weakly-succ-def} \
  \begin{enumerate}
  \item\index{amalgam} For $M_0 \lea M_\ell$ all in $\K_\lambda$, $\ell = 1,2$, an \emph{amalgam of $M_1$ and $M_2$ over $M_0$} is a triple $(f_1, f_2, N)$ such that $N \in \K_\lambda$ and $f_\ell : M_\ell \xrightarrow[M_0]{} N$.
  \item\index{equivalence of amalgam} Let $(f_1^x, f_2^x, N^x)$, $x = a,b$ be amalgams of $M_1$ and $M_2$ over $M_0$. We say $(f_1^a, f_2^a, N^a)$ and $(f_1^b, f_2^b, N^b)$ are \emph{equivalent over $M_0$} if there exists $N_\ast \in \K_\lambda$ and $f^x : N^x \rightarrow N_\ast$ such that $f^b \circ f_1^b = f^a \circ f_1^a$ and $f^b \circ f_2^b = f^a \circ f_2^a$, namely, the following commutes:

  
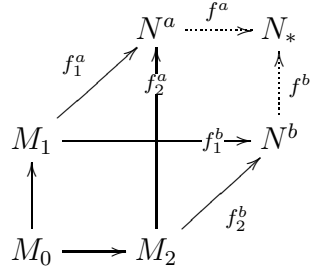
\[
  \xymatrix{ & N^a \ar@{.>}[r]^{f^a} & N_\ast \\
    M_1 \ar[ru]^{f_1^a} \ar[rr]|>>>>>{f_1^b} & & N^b \ar@{.>}[u]_{f^b} \\
    M_0 \ar[u] \ar[r] & M_2 \ar[uu]|>>>>>{f_2^a}  \ar[ur]_{f_2^b} & \\
  }
  \]

  Note that being ``equivalent over $M_0$'' is an equivalence relation (\cite[Proposition 4.3]{jrsh875}).
\item Let $\s$ be a type-full good $\lambda$-frame on $\K$.
  \begin{enumerate}
    \item Let $\Ktuqp{\s}$ denote the set of triples $(a, M, N)$ such that $M \lea N$ are in $\K_\lambda$, $a \in |N| \backslash |M|$ and for any $M_1 \gea M$ in $\K_\lambda$, there exists a \emph{unique} (up to equivalence over $M$) amalgam $(f_1, f_2, N_1)$ of $N$ and $M_1$ over $M$ such that $\gtp (f_1 (a) / f_2[M_1] ; N_1)$ does not fork over $M$. We call the elements of $\Ktuqp{\s}$ \emph{uniqueness triples}. When $\s$ is clear from context, we just write $\Ktuq$.
    \item$\Ktuqp{\s}$ has the \emph{existence property} if for any $M \in \K_{\lambda}$ and any nonalgebraic $p \in \gS (M)$, one can write $p = \gtp (a / M; N)$ with $(a, M, N) \in \Ktuqp{\s}$.
    \item We say that \emph{$\s$ has the existence property for uniqueness triples} or \emph{$\s$ is weakly successful} if $\Ktuqp{\s}$ has the existence property.
  \end{enumerate}
  \end{enumerate}
\end{defin}

\begin{remark}\label{amalgam-rmk}
  Let $M_0 \lea M_1, M_0 \lea M_2$ all be in $\K_\lambda$.
  \begin{enumerate}
  \item If $(f_1, f_2, N)$ is an amalgam of $M_1$ and $M_2$ over $M_0$, there exists an equivalent amalgam $(g_1, g_2, N')$ of $M_1$ and $M_2$ over $M_0$ with $g_2 = \operatorname{id}_{M_2}$.
  \item For $x = a, b$, assume that $(f_1^x, f_2^x, N^x)$ are \emph{non-equivalent} amalgams of $M_1$ and $M_2$ over $M_0$. We have the following monotonicity properties:
    \begin{enumerate}
    \item For $x = a,b$, if $N^x \lea N_*^x$, then $(f_1^x, f_2^x, N_*^x)$ are non-equivalent amalgams of $M_1$ and $M_2$ over $M_0$.
    \item If $M_1 \lea M_1'$, $M_2 \lea M_2'$ (with $M_1', M_2' \in \K_\lambda$), and for $x = a, b$, there exists $g_1^x \supseteq f_1^x, g_2^x \supseteq f_2^x$ such that $(g_1^x, g_2^x, N^x)$ is an amalgam of $M_1'$ and $M_2'$ over $M_0$, then $(g_1^a, g_2^a, N^a)$ and $(g_1^b, g_2^b, N^b)$ are not equivalent over $M_0$.
    \end{enumerate}
  \end{enumerate}
\end{remark}

We are now ready to prove the desired result.

\begin{thm}[Shelah]\label{weakly-succ-thm}
  Assume $2^{\lambda} < 2^{\lambda^+}$. Let $\s$ be a good $\lambda$-frame on $\K$. If $\K$ is categorical in $\lambda$ and for any saturated $M \in \K_{\lambda^+}$ there exists $N \in \K_{\lambda^+}$ universal over $M$, then $\s$ is weakly successful.
\end{thm}
\begin{proof}
  Below, we assume to simplify the notation that $\s$ is type-full but the same proof goes through in the general case. Suppose that the conclusion of the theorem fails. Fix $\seq{M_\alpha : \alpha < \lambda^+}$ increasing continuous in $\K_\lambda$ such that $M_{\alpha + 1}$ is limit over $M_\alpha$ for all $\alpha < \lambda^+$. Let $M_{\lambda^+} := \bigcup_{\alpha < \lambda^+} M_\alpha$. Since $\s$ is not weakly successful, there exists a nonalgebraic type $p \in \gS (M_0)$ which cannot be represented by a uniqueness triple. Say $p = \gtp (a / M_0; M) $.
  
  We build a strictly increasing continuous tree $\seq{M_\eta : \eta \in \fct{\le \lambda^+}{2}}$ with $M_\eta \in \K_\lambda$ for $\eta \in \fct{<\lambda^+}{2}$, as well as a strictly increasing continuous tree of embeddings $\seq{f_\eta : \eta \in \fct{\le \lambda^+}{2}}$ such that for any $\eta \in \fct{<\lambda^+}{2}$:

  \begin{enumerate}
  \item $f_\eta : M_{\ell (\eta)} \rightarrow M_\eta$.
  \item\label{weak-succ-2} $M_{<>} = M$ and $f_{<>} = \operatorname{id}_{M_0}$.
  \item\label{weak-succ-4} $\gtp (a / f_\eta[M_{\ell (\eta)}]; M_\eta)$ does not fork over $M_0$.
  \item\label{weak-succ-5} There is \emph{no} $N \in \K_\lambda$ and $g_\ell : M_{\eta \smallfrown l} \rightarrow N$, $\ell = 0,1$, such that the following diagram commutes:

  \[
  \xymatrix{ & M_{\eta \smallfrown 0} \ar@{.>}[r]^{g_0} & N \\
    M_{\ell (\eta) + 1} \ar[ru]^{f_{\eta \smallfrown 0}} \ar[rr]|>>>>>{f_{\eta \smallfrown 1}} & & M_{\eta \smallfrown 1} \ar@{.>}[u]_{g_1} \\
    M_{\ell (\eta)} \ar[u] \ar[r]_{f_\eta} & M_\eta \ar[uu] \ar[ur] & \\
  }
  \]    
  \end{enumerate}

  \underline{This is enough}: We have that $M_{\lambda^+}$ is saturated and we know by assumption that there is a universal model $N$ over $M_{\lambda^+}$ in $\K_{\lambda^+}$. In particular, for every $\eta \in \fct{\lambda^+}{2}$, there exists $g_\eta : M_\eta \rightarrow N$ such that $f_\eta^{-1} \subseteq g_\eta$. By Lemma \ref{weak-diamond-lem}, requirement (\ref{weak-succ-5}) must fail somewhere in the construction, contradiction.

  \underline{This is possible}: The construction is by induction on the length of $\eta \in \fct{\le \lambda^+}{2}$. If $\ell (\eta) = 0$, then (\ref{weak-succ-2}) specifies what to do and if the length is limit, we take unions (and use the local character and transitivity properties of forking to see that (\ref{weak-succ-4}) is preserved). Assume now that $\alpha < \lambda^+$ and that $\eta \in \fct{\alpha}{2}$ is such that $M_\eta$, $f_\eta$ have been defined. We want to build $M_{\eta \smallfrown \ell}$, $f_{\eta \smallfrown \ell}$ for $\ell = 0,1$. Let $q := \gtp (a / f_{\eta}[M_{\alpha}]; M_\eta)$. We know that $q$ is the nonforking extension of $p$ (by (\ref{weak-succ-4}) and the definition of $p$), so by the conjugation property (Fact \ref{conj-prop}, note that by assumption $\K$ is categorical in $\lambda$) $p$ and $q$ are conjugates, hence $q$ cannot be represented by a uniqueness triple. Therefore $(a, f_{\eta}[M_{\alpha}], M_\eta)$ is not a uniqueness triple. This means that there exists $M_\alpha' \in \K_\lambda$ with $f_\eta[M_\alpha] \le M_{\alpha}'$ and two non-equivalent amalgams $(f_1^\ell, f_2^\ell, M_{\eta \smallfrown \ell})$ such that $\gtp (f_2^\ell (a) / f_1^\ell[M_{\alpha}']; M_{\eta \smallfrown \ell})$ does not fork over $M_\eta$ for $\ell = 0, 1$. Without loss of generality (Remark \ref{amalgam-rmk}) $f_2^\ell$ is the identity for $\ell = 0,1$:

  \[
  \xymatrix{ & M_{\eta \smallfrown 0}  &  \\
    M_{\alpha}' \ar[ru]^{f_1^0} \ar[rr]|>>>>>{f_{1}^1} & & M_{\eta \smallfrown 1}  \\
    f_\eta[M_{\alpha}] \ar[u] \ar[r] & M_\eta \ar[uu] \ar[ur] & \\
  }
  \]

  By the monotonicity property of being non-equivalent amalgams (Remark \ref{amalgam-rmk}) and the extension property of forking, we can increase $M_\alpha'$, $M_\eta^0$, and $M_\eta^1$ to assume without loss of generality that $M_{\alpha}'$ is limit over $f_\eta[M_{\alpha}]$. In particular, there exists $g: M_{\alpha + 1} \cong M_\alpha'$ with $f_\eta \subseteq g$. For $\ell = 0,1$, let $f_{\eta \smallfrown \ell} := f_1^\ell \circ g$.  
\end{proof}

\begin{proof}[Proof of Fact \ref{weakly-succ-fact}]
  Note that amalgamation and stability in $\lambda^+$ imply that over every $M \in \K_{\lambda^+}$ there exists $N \in \K_{\lambda^+}$ universal over $M$. Thus the hypotheses of Theorem \ref{weakly-succ-thm} hold.
\end{proof}

\bibliographystyle{amsalpha}
\bibliography{categ-good-aecs}

\end{document}